\documentclass[letterpaper, 10 pt, conference]{ieeeconf}  

\IEEEoverridecommandlockouts                              
\usepackage{graphics} 
\usepackage{epsfig} 
 \usepackage{mathptmx} 
\usepackage{amsmath} 
\usepackage{amssymb}  
\usepackage{subfigure}
\usepackage{xcolor}
\usepackage{enumerate}

\newcommand{\contTilde}[1]{\mathbf{\tilde{#1}}}
\newcommand{\transpose}{\mathsf{T}}

\newcommand{\quadinner}[1]{x^{\transpose}(#1)x}
\DeclareMathOperator{\contB}{\mathbf{B}}

\newcommand{\BK}[1]{\mathbf{B}\bar{K}_{#1}}

\newtheorem{lemma}{Lemma}
\newtheorem{assumption}{Assumption}
\newtheorem{definition}{Definition}
\newtheorem{corollary}{Corollary}
\newtheorem{proposition}{Proposition}
\newtheorem{remark}{Remark}
\newtheorem{theorem}{Theorem}

\title{\LARGE \bf
Two-Player Dynamic Potential LQ Games with Sequentially Revealed Costs
}

\author{ \parbox{3 in}{\centering Yitian Chen, Timothy L. Molloy, Iman Shames
        \thanks{The authors are with CIICADA Lab, School of Engineering,        Australian National University (ANU) \texttt{\{yitian.chen,timothy.molloy,iman.shames\}@anu.edu.au
        }}\\
        }
}

\allowdisplaybreaks

\begin{document}

\maketitle
\thispagestyle{empty}
\pagestyle{empty}

\begin{abstract}
We investigate a novel finite-horizon linear-quadratic (LQ) feedback dynamic potential game with a priori unknown cost matrices played between two players.
The cost matrices are revealed to the players sequentially, with the potential for future values to be previewed over a short time window. 
We propose an algorithm that enables the players to predict and track a feedback Nash equilibrium trajectory, and we measure the quality of their resulting decisions by introducing the concept of \emph{price of uncertainty}.
We show that under the proposed algorithm, the price of uncertainty is bounded by horizon-invariant constants.
The constants are the sum of three terms; the first and second terms decay exponentially as the preview window grows, and another depends on the magnitude of the differences between the cost matrices for each player. 
Through simulations, we illustrate that the resulting price of uncertainty initially decays at an exponential rate as the preview window lengthens, then remains constant for large time horizons.
\end{abstract}

\section{INTRODUCTION}
Noncooperative dynamic game theory is a mathematical framework for decision-making among rational players in dynamic environments \cite{basar_dynamic_1998,engwerda_lq_2005}.
It has proven invaluable for modelling interactions between agents in applications including networked controls and communications \cite{perin_static_2023,thanh_van_dynamic_2022}, economics \cite{wu_real-time_2022,yan_stability_2022} and power systems \cite{dong_non-cooperative_2022,scarabaggio_noncooperative_2022,serna-torre_non-cooperative_2024,wu_market_2024}.

\emph{Dynamic Potential LQ Games} are a class of dynamic games in which feedback Nash equilibria can be determined by solving multivariate optimal control problems \cite{prasad_structure_2023}.
The recent development of dynamic potential LQ games has improved the tractability of noncooperative dynamic game models in many applications by enabling the use of well-established optimal-control solutions techniques to find Nash equilibria \cite{zazo_dynamic_2016, prasad_structure_2023}.
The majority of dynamic potential LQ games are, however, formulated assuming that players have complete knowledge of the costs incurred by themselves and others.
In many practical situations, the costs that players could incur are unknown in advance.
We therefore examine dynamic potential LQ games with \emph{a priori} unknown cost matrices.

Our practical motivation for considering dynamic potential games with \emph{a priori} unknown costs is tied to the integration of renewable energy and community batteries into domestic energy markets.
For example, dynamic potential LQ game theory has been used to model and analyse scheduling energy usage for multiple users over a period of time, with each user seeking to minimise generation costs from renewable energy sources \cite{wu_demand_2011}.
Renewable generation costs are time-varying due to their dependency on weather conditions and can potentially be predicted only over a short future horizon. 
However, existing models do not account for the sequential revelation of costs.
Motivated by the above problem, we consider dynamic potential LQ games with sequentially revealed pay-offs (and a more general class of linear systems). 
The analysis in our work opens the possibility of more comprehensive models of decentralised decision-making in energy markets that accounts for renewable energy availability. 

From a theoretical standpoint, our consideration of dynamic potential games with sequentially revealed costs extends recent \emph{online} LQ optimal control problems \cite{chen_regret_2023, zhang_regret_2021, cohen_online_2018, akbari_logarithmic_2022} from a single-player setting to a multi-player noncooperative setting.
The extension is, however, nontrivial. 
In online LQ control problems, the concept of \emph{regret} is used to measure suboptimality of an algorithm against the optimal solution in hindsight.  
In a dynamic game setting, similar concepts of best-performance-in-hindsight are more difficult to define because the concept of optimality is itself ambiguous in noncooperative competitive settings, leading to equilibria solution concepts.
Similarly, any proposed algorithms for solving online dynamic games must be tailored to the specific solution concept.
To the best of our knowledge, these challenges have not been overcome in generalising results and algorithms from online LQ control problems to online LQ dynamic game problems. Importantly, while other decision quality measures have been explored in Markov games \cite{narasimha_multi-agent_2022, leonardos_global_2022, tian_online_2021}, these notions offer limited insight into LQ settings with continuous state and action spaces.

We consider feedback Nash equilibria as the solution concept of interest and introduce a price of uncertainty (PoU) performance measure to capture deviations of online algorithms from the value of Nash equilibria of the game computed in hindsight. 
Motivated by the concept of the price of uncertainty (PoU) \cite[Section 3.2]{grossklags_price_2010} \cite{basar_prices_2011} in static games, which captures discrepancies between expected payoffs in environments with complete and incomplete information, our PoU quantifies the difference between a decision made with causal cost information and a policy that achieves a feedback Nash equilibrium in hindsight. This is done by evaluating the difference between their respective values. In addition, in our case, the expected payoff under complete cost information is determined by the average cost incurred at the feedback Nash equilibrium, which represents the incurred cost by the decision naturally made by players in the noncooperative game setting.

In general, the existence and uniqueness of a Nash equilibrium in a noncooperative dynamic game are not guaranteed, even if the cost functions of all players are strictly convex \cite[Chapter 6, Section 6.2.2]{basar_dynamic_1998}. 
However, we shall impose specific conditions on the parameters of LQ dynamic feedback potential games so that we may define both our algorithm and performance guarantees with respect to a unique feedback Nash equilibrium.

The key contributions of this paper are:
\begin{enumerate}
    \item The formulation of a novel two-player LQ dynamic feedback potential game problem with sequentially revealed costs and a price of uncertainty (PoU) performance measure;
    \item The proposal of an algorithm that enables both players to predict and track a feedback Nash equilibrium in a two-player LQ dynamic feedback potential game with sequentially revealed costs;
    \item The lower and upper bound on the PoU achieved by our proposed algorithm;
    \item Numerical simulation on a modified electricity usage schedule problem based on \cite{wu_demand_2011}. Results show that the absolute value of PoU initially decays exponentially as the preview window increases, then remains constant for large time horizons, which matches our theoretical lower and  upper bounds.
\end{enumerate}

This paper is structured as follows.
In Section \ref{sec:problem_formulation}, we formulate the problem of two-player LQ dynamic feedback potential game with sequentially revealed costs. In Section \ref{sec:approach}, we present our proposed algorithm and an upper bound on its price of uncertainty. In Section \ref{sec:numerical}, we illustrate the performance of the proposed algorithm applied to a two-player electricity usage scheduling problem. We present concluding remarks and future directions in Section \ref{sec:conclusions}.

\paragraph*{Notation}
The $ij$-th block of a block matrix $\mathbf{H}$ is
denoted by $[\mathbf{H}]_{ij}$.
Moreover, for $i \in \{1,2\}$, we define the following notation for a (dual-indexed) sequence
$(\pi_{t}^{i})_{t=1}^{T-1} := (\pi_{i,1},\cdots, \pi_{i,T-1})$ and $(\pi_{t}^{i})_{i=1,t=1}^{2,T-1} := ((\pi_{t}^{1})_{t=1}^{T-1},(\pi_{t}^{2})_{t=1}^{T-1})$.
We use $\|\cdot\|$ to denote the 2-norm of a vector or a matrix, depending on
its argument. For symmetric matrices $F,G$ with appropriate dimensions, $F
\preceq G$ if and only if $G-F$ is positive semi-definite. Let $\rho(\cdot)$ be the spectral radius operator. For a positive integer $k$, define $\mathbb{N}_k$
as the set of integers $\{1,2,\dots,k\}$. We use $\lambda_{n}(\cdot)$,
$\lambda_{min}(\cdot)$ and $\lambda_{max}(\cdot)$ to designate the $n$-th, the
minimum and the maximum eigenvalue of a matrix, respectively. Similarly for the
singular value operator, $\sigma_{n}(\cdot)$, $\sigma_{min}(\cdot)$ and
$\sigma_{max}(\cdot)$ denote the $n$-th, the smallest, and the largest singular values of a matrix. Moreover, let $\sigma_{min}^{+}(\cdot)$ and $\sigma_{max}^{+}(\cdot)$ be the operators for minimum and maximum non-zero
singular value, respectively. Let $\mathbb{S}_{++}^{n}$ denote the set of $n\times n$ positive definite matrix and $\mathbb{S}_{+}^{n}$ denote the set of $n\times n$ positive semi-definite matrix. 

\section{PROBLEM FORMULATION}
\label{sec:problem_formulation}

In this section, we formulate a two-player LQ dynamic feedback potential game with sequentially revealed costs, and provide the technical assumptions under which we shall develop our associated algorithm and PoU analysis.

\subsection{Two-Player LQ Dynamic Feedback Potential Game with Sequentially Revealed Costs}\label{section:LQ-DFG}
Let us first define a (noncooperative) two-player LQ dynamic feedback game (LQ-DFG).
The state $x_t \in \mathbb{R}^n$ of the game evolves according to the discrete-time linear system
\begin{align}
    \label{eq:linsys}
    x_{t+1} &= Ax_{t} + \mathbf{B}\mathbf{u}_{t}, \quad x_{1} = \bar{x}_{1},
\end{align}
where $t$ is a non-negative integer, $A \in \mathbb{R}^{n\times n}$ and $\mathbf{B} = [B^{1}, B^{2}]$, $B^{1},B^{2} \in \mathbb{R}^{n\times m}$ are \emph{ system matrices}, $\mathbf{u}_{t} = [u_{t}^{1\transpose}, u_{t}^{2\transpose}]^{\transpose}$ where $u_{t}^{1}, u_{t}^{2} \in \mathbb{R}^{m}$ are player controls, and $\bar{x}_{1} \in \mathbb{R}^{n}$ is the initial state.
Let $\Lambda$ denote the set of all measurable maps from $\mathbb{R}^{n}$ to $\mathbb{R}^{m}$ and $\Pi_{t}:= (\pi_{t}^{1},\pi_{t}^{2})$.
Each player $i$, $i \in \{1,2\}$, selects a \emph{feedback control policy} $\pi_{t}^{i} \in \Lambda$, to determine their controls according to $u_{t}^{i} = \pi_{t}^{i}(x_{t})$ with the aim of minimising a quadratic cost function 
\begin{align}\label{eq:LQcost}
    \begin{split}
        &J_{i,T}(\bar{x}_{1}, (\Pi_{t})_{t=1}^{T-1}) := \sum_{t=1}^{T-1} x_{t+1}^{\transpose}Q_{t+1}x_{t+1} + \mathbf{u}_{t}^{\transpose}R_{t}^{i}\mathbf{u}_{t}
    \end{split}
\end{align}
subject to \eqref{eq:linsys} where $\mathbf{u}_{t} = [u_{t}^{1\transpose}, u_{t}^{2\transpose}]^{\transpose} = [\pi_{t}^{1}(x_{t})^{\transpose}, \pi_{t}^{2}(x_{t})^{\transpose}]^{\transpose} = \Pi_{t}(x_{t})$, $Q_{t+1}\in\mathbb{S}_{++}^n$ and 
$R_{t}^{i}$ admits the following block structure:
    \begin{align}\label{eq:rblock}
   &R_{t}^{i} := 
   \begin{bmatrix}
       [R_{t}^{i}]_{11} & [R_{t}^{i}]_{12}\\
       [R_{t}^{i}]_{21} & [R_{t}^{i}]_{22}
   \end{bmatrix}\in \mathbb{R}^{2m\times 2m}\text{ with }[R_{t}^{i}]_{jk}\in \mathbb{R}^{m\times m}.
   \end{align}
for all $t \in \mathbb{N}_{T-1}$ and $j,k\in \{1,2\}$.
The players do not (explicitly) cooperate, and thus we shall consider the solution that arises to be a \emph{feedback Nash equilibrium}. For feedback Nash equilibrium policies $\pi_{t}^{1*}, \pi_{t}^{2*}$, let $\Pi_{t}^{*}:=(\pi_{t}^{1*}, \pi_{t}^{2*})$.
Specifically, the feedback control policies $(\pi_{t}^{1*})_{t = 1}^{T-1}$ and $(\pi_{t}^{2*})_{t = 1}^{T-1}$ with $\pi_{t}^{1*},\pi_{t}^{2*} \in \Lambda$, constitute a \emph{feedback Nash equilibrium} if, for any given $t\in \mathbb{N}_{T-1}$ and $i\in \{1,2\}$, they satisfy the pair of inequalities \cite[Equation (3.27)-(3.28)]{basar_dynamic_1998} 
\begin{align}\label{eq:nashIneq}
    &J_{T}^{1}(\bar{x}_{1}, (\Pi_{\tau})_{\tau=1}^{t-1}, \Pi_{t}^{*},(\Pi_{\tau}^{*})_{\tau=t+1}^{T-1})\notag\\
    &\qquad \leq J_{T}^{1}(\bar{x}_{1}, (\Pi_{\tau})_{\tau=1}^{t-1}, (\pi_{t}^{1},\pi_{t}^{2*}),(\Pi_{\tau}^{*})_{\tau=t+1}^{T-1}),\\
    &J_{T}^{2}(\bar{x}_{1}, (\Pi_{\tau})_{\tau=1}^{t-1}, \Pi_{t}^{*},(\Pi_{\tau}^{*})_{\tau=t+1}^{T-1})\notag\\
    &\qquad \leq J_{T}^{1}(\bar{x}_{1}, (\Pi_{\tau})_{\tau=1}^{t-1}, (\pi_{t}^{1*},\pi_{t}^{2}),(\Pi_{\tau}^{*})_{\tau=t+1}^{T-1}),\notag
\end{align}
for any feedback control policies $(\pi_{t}^{1})_{t=1}^{T-1}$ and $(\pi_{t}^{2})_{t=1}^{T-1}$ with $\pi_{t}^{1},\pi_{t}^{2} \in \Lambda$. 
The state dynamics \eqref{eq:linsys}, player cost functions \eqref{eq:LQcost}, and feedback Nash equilibrium solution concept \eqref{eq:nashIneq}, together define a two-player LQ-DFG.
For notational convenience, let us define $\mathrm{DFG}$ as an operator that returns state and control sequences under a feedback Nash equilibrium solution of this LQ-DFG, i.e., 
\begin{align}\label{eq:DFG}
    ((x_{t}^{*})_{t=1}^{T},(\mathbf{u}_{t}^{*})_{t=1}^{T-1})
    = \mathrm{DFG}(\mathcal{I},T,A,\contB)
\end{align}
where $\mathbf{u}_{t}^{*} = \Pi_{t}^{*}(x_{t}^{*})$ and $\mathcal{I} := (\bar{x}_{1}, (Q_{\tau+1}, R_{\tau}^{1},R_{\tau}^{2})_{\tau=1}^{T-1}))$ is the initial state and cost-function information necessary to compute a feedback Nash equilibrium. 
There is a close relationship between LQ-DFGs, and linear-quadratic optimal control problems (LQ-OCPs), with LQ-OCPs defined as follows.
\begin{definition}[LQ-OCP]\label{def:LQOCP}
    An LQ-OCP is the problem of finding a policy $(\bar{\Pi}_{t})_{t=1}^{T-1}$ solving
    \begin{align}
    \min_{(\bar{\Pi}_{t})_{t=1}^{T-1}}\qquad & \sum_{t=1}^{T-1} x_{t+1}^{\transpose}\bar{Q}_{t+1}x_{t+1} + \mathbf{u}_{t}^{\transpose}\bar{R}_{t}\mathbf{u}_{t}\label{eq:LQ-OCP}\\
    \text{s.t.} \qquad & \mathbf{u}_{t} = \bar{\Pi}_{t}(x_{t}), \text{ and } \eqref{eq:linsys}. \notag
\end{align}
for a given positive integer $T \geq 1$, and positive definite cost matrices $(\bar{Q}_{t})_{t=1}^{T}$ and $(\bar{R}_{t})_{t=1}^{T-1}$.
\end{definition}
\begin{remark}\label{remark:optimalControl}
    A special case of the LQ game is  $[R_{t}^{1}]_{11} = [R_{t}^{2}]_{22}$, $[R_{t}^{1}]_{12} = [R_{t}^{2}]_{21}$ and $B^{1} = B^{2}$. In this scenario, the LQ dynamic game reduces to an optimal control problem. In our work, we consider a more general setting beyond this optimal control case.
\end{remark}
 We specifically leverage a recent result on the connection between LQ-DFGs and LQ-OCPs that gives rise to the following recently developed concept of linear-quadratic dynamic feedback potential games (LQ-DFPGs).
\begin{definition}[LQ-DFPG \cite{prasad_structure_2023}]\label{def:LQ-DFPG}
    An LQ-DFG is referred to as an LQ-DFPG, if there exists an LQ-OCP such that the solution of the LQ-OCP is a feedback Nash equilibrium of the LQ-DFG.
\end{definition}

When the cost-matrices are known \emph{a priori} to the players, the (Nash-equilibrium) solution of an LQ-DFPG can be found in closed form (cf.\ \cite{prasad_structure_2023} and \cite[Chapter 6]{basar_dynamic_1998}).
However, in practice full information about the cost matrices over the whole time horizon $T$ may be unavailable to the players in advance.
Hence, in this paper, we suppose that at any time $t \in \mathbb{N}_{T-1-W}$ where $W\in\mathbb{N}_{T-1} \cup \{0\}$, only the initial condition of the system \eqref{eq:linsys} and the (partial) sequences of cost matrices $(Q_{\tau+1})_{\tau=1}^{t+W}$ and $(R_{\tau}^i)_{i=1,\tau=1}^{2,t+W}$ are known to the players.
Let the cost information available to the players at time $t$ be
\begin{align}\label{eq:history}
    \mathcal{H}_{t} := (\bar{x}_{1}, (Q_{\tau+1}, R_{\tau}^{1},R_{\tau}^{2})_{\tau=1}^{t+W}).
\end{align}
Our focus will be to propose a novel control policy for each player in an LQ-DFPGs that uses the information available to them only at time $t$, and to investigate the resulting deviation in performance from a feedback Nash equilibrium by using the novel concept of price of uncertainty in dynamic games.
For $i \in \{1,2\}$, we specifically consider player-feedback control policies $\pi_{t}^{i} (\cdot, \cdot)$ of the form 
\begin{align}\label{eq:feedbackForm}
    u_{t}^{i} = \pi_{t}^{i}(x_t, \mathcal{H}_t),
\end{align}
and adopt the following notion of the price of uncertainty.
\begin{definition}[Price of Uncertainty]
Given an LQ-DFPG, let $\Lambda_{\text{NE}}^{*}$ be the set of all feedback Nash equilibrium policies. For any sequence of player control policy $(\Pi_{t})_{t=1}^{T-1}$ that relies only on the available information $\mathcal{H}_{t}$ in form of \eqref{eq:feedbackForm}, we define the price of uncertainty (PoU) as
\begin{align}\label{eq:regret}
    &\text{PoU}_{T}((\Pi_{t})_{t=1}^{T-1})\notag\\
    &\!\!:=\!\! \frac{1}{2}\!\sum_{i=1}^{2}\! J_{i,T}(\bar{x}_{1},(\Pi_{t})_{t=1}^{T-1})\!-\!\!\!\! \min_{(\Pi_{t}^{*})_{t=1}^{T-1}\in \Lambda_{\text{NE}}^{*}}\!\!\frac{1}{2}\sum_{i=1}^{2}J_{i,T}(\bar{x}_{1},(\Pi_{t}^{*})_{t=1}^{T-1}).
\end{align}
\end{definition}
The notion of PoU above measures the maximum possible gap between the player control policies and the feedback Nash equilibrium policies determined in hindsight, therefore we pick $(\Pi_{t}^{*})_{t=1}^{T-1}$ that achieves the minimal average cost among all $\Lambda_{\text{NE}}^{*}$.
In our work, we specifically consider the scenario in which the LQ-DFG has a unique feedback Nash equilibrium. Thus, 
\begin{align}\label{eq:uniquePoU}
    &\text{PoU}_{T}((\Pi_{t})_{t=1}^{T-1})\notag\\
    &\!\!=\!\! \frac{1}{2}\!\sum_{i=1}^{2}\! J_{i,T}(\bar{x}_{1},(\Pi_{t})_{t=1}^{T-1})\!-J_{i,T}(\bar{x}_{1},(\Pi_{t}^{*})_{t=1}^{T-1}).
\end{align}
Details of the scenario are given in the next section.

\begin{remark}
     The price of uncertainty coincides with
\emph{dynamic regret} in online optimal control (cf. \cite{chen_regret_2023}) when there is only a single player. In other words, in this scenario, PoU reduces to \emph{dynamic regret} in \cite[Equation (9)]{chen_regret_2023}. 
\end{remark}
We aim to show that the PoU associated with our proposed policy is sublinear with respect to the time horizon $T$, i.e., $\text{PoU}_{T}((\Pi_{t})_{t=1}^{T-1})= o(T)$, where $o(T)$ satisfies $\lim_{T\rightarrow \infty} \frac{o(T)}{T} = 0$. This implies that, on average, the algorithm performs as well as the strategy that leads to feedback Nash equilibrium in hindsight, evaluated from the perspective of average costs.


\subsection{Technical Assumptions and Preliminary Results}\label{sec:assumptions}
We develop our algorithm and analysis under the following assumptions on the systems structure and the cost matrices that ensure we are considering LQ-DFPGs. 

We first introduce an assumption that an LQ-DFG with parameters $A,\contB,(Q_{t})_{t=1}^{T},(R_{t}^{1},R_{t}^{2})_{t=1}^{T-1}$ that satisfy the conditions given in \cite[Theorem 6]{prasad_structure_2023} is sufficient to constitute an LQ-DFPG. Moreover, the invertibility of the matrices $\Theta_{t}$ in \eqref{eq:costFPDG2} guarantees the uniqueness of feedback Nash Equilibria in the LQ dynamic game, regardless of whether the game is potential \cite[Remark 6.5, Chapter 6.2.2]{basar_dynamic_1998}. This assumption ensures that there is only a single policy available in the set $\Lambda_{\text{NE}}^{*}$.

\begin{assumption}\label{assumption:parameters}
        For $t\in \mathbb{N}_{T-1},i,j\in \{1,2\}$,
    define
        \begin{align}\label{eq:Theta}
        [\Theta_{t}]_{ij} := [R_{t}^{i}]_{ij} + B^{i\transpose}P_{t+1}^{i}B^{j},
        \end{align}
    \begin{align}\label{eq:costFPDG4}
        P_{T}^{i} =Q_{T}.
    \end{align}
    Assume
    \begin{align}\label{eq:costFPDG2}
        Q_{t} \in \mathbb{S}_{++}^{n}, \Theta_{t} \in \mathbb{S}_{++}^{2m},
    \end{align}
    for all $t \in \mathbb{N}_{T}$.
    Furthermore, define
    \begin{align}\label{eq:gainK}
        K_{t} = -\Theta_{t}^{-1}
        \begin{bmatrix}
            B^{1}P_{t+1}^{1}\\
            B^{2}P_{t+1}^{2}
        \end{bmatrix}A,
    \end{align}
    \begin{align}\label{eq:pIteration}
        P_{t}^{i} = Q_{t} + K_{t}^{\transpose}R_{t}^{i}K_{t}+ (A+BK_{t})^{\transpose}P_{t+1}^{i}(A+BK_{t}).
    \end{align}
    We further assume the following,
    \begin{align}\label{eq:costFPDG1}
        [R_{t}^{1}]_{12} + B^{1\transpose}P_{t+1}^{1}B^{2} = ([R_{t}^{2}]_{21} + B^{2\transpose}P_{t+1}^{2}B^{1})^{\transpose},
    \end{align}
    
        \begin{align}\label{eq:costFPDG3}
\mathbf{B}^{\transpose}P_{t}^{1}A=\mathbf{B}^{\transpose}P_{t}^{2}A.
        \end{align}
\end{assumption}
The next lemma connects the LQ-DFG parameters defined in the above assumption to the concept of LQ-DFPG. 
\begin{lemma}
    Consider the LQ-DFG with parameters $A,\contB,(Q_{t})_{t=1}^{T},(R_{t}^{1},R_{t}^{2})_{t=1}^{T-1}$ that satisfy Assumption \ref{assumption:parameters}, then
    \begin{enumerate}[i)]
        \item
        The LQ-DFG is an LQ-DFPG
        \item 
        The LQ-DFPG has a unique feedback Nash equilibria
        \item 
        For any player control policy $(\Pi_{t})_{t=1}^{T-1}$, the notion of PoU reduces to \eqref{eq:uniquePoU}.
    \end{enumerate}
\end{lemma}
\begin{proof}
    For i) and ii), please see the proof of \cite[Theorem 6]{prasad_structure_2023}. Further, we can deduce iii) from i) and ii).
\end{proof}

We make the following assumption for the cost matrices.
\begin{assumption}\label{assumption:bounds}
    For $t\in \mathbb{N}_{T-1}, i\in \{1,2\}$, $Q_{t+1}\in \mathbb{S}_{++}^{n}$ and $R_{t}^{i} \in \mathbb{S}_{+}^{m}$. Further, there exist $Q_{min}, Q_{max}\in \mathbb{S}_{++}^{n}$ and $R^{'}_{min}, R^{'}_{max}\in \mathbb{S}_{+}^{m}$, such that,
    \begin{align*}
        Q_{min} &\preceq Q_{t+1} \preceq Q_{max},
    \end{align*}
    \begin{align}\label{eq:positiveR}
        R^{'}_{min} \preceq R_{t}^{i} \preceq R^{'}_{max}.
    \end{align}
\end{assumption}

We also require the following (mild) assumption on the system dynamics \eqref{eq:linsys}.

\begin{assumption}\label{assumption:controllable}
    The state transition matrix $A$ from \eqref{eq:linsys} has full rank, and there exists a $\bar{K}$, such that $\rho(A + \contB \bar{K}) < 1$.
\end{assumption}
The above assumption ensures that there exists a stabilizing controller $\bar{K}$ for the system $(A,\mathbf{B})$.
The next corollary reveals property associated with cost matrices $R_{t}^{i}$.
\begin{lemma}\label{lemma:symmetryR}
    Under Assumption \ref{assumption:controllable}, the matrix 
    \begin{align}\label{eq:potentialR}
        R_{t}^{p} :=\begin{bmatrix}
            [R_{t}^{1}]_{11} & [R_{t}^{1}]_{12}\\
            [R_{t}^{2}]_{21} & [R_{t}^{2}]_{22}
        \end{bmatrix}
    \end{align}
    is symmetric for $t\in \mathbb{N}_{T-1}$.
\end{lemma}
\begin{proof}
    Under Assumption \ref{assumption:controllable}, \eqref{eq:costFPDG1} and \eqref{eq:costFPDG3} imply $ [R_{t}^{1}]_{12}-[R_{t}^{2}]_{21}^{\transpose} + B^{1\transpose}(P_{t+1}^{1}-P_{t+1}^{2})B^{2} = [R_{t}^{1}]_{12}-[R_{t}^{2}]_{21}^{\transpose}=0$. Moreover, under Assumption \ref{assumption:bounds}, $[R_{t}^{1}]_{11}$ and $[R_{t}^{2}]_{22}$ are symmetric. Therefore,
    \begin{align*}
        \begin{bmatrix}
            [R_{t}^{1}]_{11} & [R_{t}^{1}]_{12}\\
            [R_{t}^{2}]_{21} & [R_{t}^{2}]_{22}
        \end{bmatrix}^{\transpose} = 
        \begin{bmatrix}
            [R_{t}^{1}]_{11}^{\transpose} & [R_{t}^{2}]_{21}^{\transpose}\\
             [R_{t}^{1}]_{12}^{\transpose}& [R_{t}^{2}]_{22}^{\transpose}
        \end{bmatrix} =  \begin{bmatrix}
            [R_{t}^{1}]_{11} & [R_{t}^{1}]_{12}\\
            [R_{t}^{2}]_{21} & [R_{t}^{2}]_{22}
        \end{bmatrix}.
    \end{align*}
    Thus, $R_{t}^{p}$ is symmetric.
\end{proof}

We also make the following assumptions on cost matrices.
\begin{assumption}\label{assumption:positiveRp}
    Matrix $R_{t}^{p}$ in \eqref{eq:potentialR} is positive definite.
\end{assumption}

\begin{assumption}\label{assumption:lowerQ}
    For any given $t \in \mathbb{N}_{T-1}, i,j\in \{1,2\}$, 
    let $q := \frac{\sigma_{\max}^{+}(A)}{\sigma_{\min}^{+}(\contB)}\lambda_{max}(R_{t}^{p}-R_{t}^{1})$.
    Matrices $Q_{t}$ and $R_{t}^{p}$ satisfy $\lambda_{min}(Q_{t}) > |q|$.
\end{assumption}

Assumptions \ref{assumption:bounds} and \ref{assumption:positiveRp} ensure that the eigenvalues for all cost matrices are positive and finite. This guarantees that all cost matrices in the LQ-OCP associated with the LQ-DFG are bounded, which coincides with standard assumption in LQR control literature, e.g., \cite[Assumption 1]{chen_regret_2023}, \cite[Assumption 1]{sun_receding-horizon_2023}. See Corollary \ref{corollary:boundedR}, and \ref{corrolary:boundedK} in the Appendix. 
Assumption \ref{assumption:lowerQ} ensures that the cost matrices in the LQ-OCP associated with the LQ-DFG are positive definite, see Corollary \ref{corollary:positiveQ} from the Appendix. 

To state the final assumption, we need to define the following parameters to help us simplify presentations. For a given preview horizon $W\in \mathbb{N}_{T-1} \cup \{0\}$, define
    \begin{equation}\label{eq:informationQ}
        Q_{\tau+1|t}:= 
        \begin{cases}
            Q_{\tau+1} \text{ if $1\leq \tau \leq t+W$}\\
            Q_{t+W+1} \text{ if $t+W < \tau \leq T-1$},
        \end{cases}
    \end{equation}
    and
    \begin{equation}\label{eq:informationR}
        R_{\tau|t}^{i}:= 
        \begin{cases}
            R_{\tau}^{i} \text{ if $1\leq \tau \leq t+W$}\\
            R_{t+W}^{i} \text{ if $t+W < \tau \leq T-1$},
        \end{cases}
    \end{equation}
    for $i \in \{1,2\}$ and $t\in \mathbb{N}_{T-1}$.

Our final assumption is then as follows.
\begin{assumption}\label{assumption:potentialRepeat}
    The LQ-DFG with parameters $\{Q_{\tau+1|t}\}_{\tau=1}^{T-1}$ and $\{R_{\tau|t}^{1},R_{\tau|t}^{2}\}_{\tau=1}^{T-1}$ is an LQ-DFPG for all $t\in \mathbb{N}_{T-1}$.
\end{assumption}
The above assumption ensures that when the last known cost matrices are repeated at time $t$ for the next $T-t$ stages, the LQ-DFG corresponds to an LQ-DFPG. 
A sufficient condition of an LQ-DFG that satisfies the above assumption is,
consider cost matrices $Q_{t}$ and $R_{t}^{i}$ as the form of
\begin{align}\label{eq:costSelection}
    &Q_{t} = 
    \begin{bmatrix}
        l_{t} & -d_{t}\\
        -d_{t} & 0
    \end{bmatrix},
    R_{t}^{1} = 
    \begin{bmatrix}
        r_{1,t} & 0\\
        0 & 0
    \end{bmatrix},\\
    &R_{t}^{2} =
    \begin{bmatrix}
        0 & 0\\
        0 & r_{2,t}
    \end{bmatrix},
    \mathbf{B} = 
    \begin{bmatrix}
        -b_{1} & -b_{2}\\
        0 & 0
    \end{bmatrix}\notag,
\end{align}
where $l_{t} > 0,d_{t} > 0,r_{1,t}\geq 0$ and $r_{2,t} \geq 0$.
Further, if $A$ is full rank and $\frac{b_{1}^{2}}{r_{1,t}} = \frac{b_{2}^{2}}{r_{2,t}}$ for all $t\in \{1,2\}$, then the LQ-DFG defined by the above parameters satisfies Assumptions \ref{assumption:parameters}, \ref{assumption:controllable}, \ref{assumption:positiveRp}, and \ref{assumption:potentialRepeat}. Additionally, we can meet Assumptions \ref{assumption:bounds} and \ref{assumption:lowerQ} by appropriately selecting bounds for elements in $Q_{t},R_{t}^{1}$ and $R_{t}^{2}$. We use this example to illustrate our proposed algorithm numerically in Section \ref{sec:numerical}. For the formal statement and proof, see Proposition \ref{prop:valid} in the Appendix.
 

In the next section, we present our algorithm for finding a control policy and provide the price of uncertainty lower and upper bound under such an algorithm.

\section{Proposed Approach and PoU Analysis}\label{sec:approach}
In this section, we propose an algorithm for solving our proposed LQ-DFPG problem and analyse its PoU.

\subsection{Algorithm}

Our algorithm involves two steps at each time step $t$ (for each player) --- a prediction step and a tracking step.

\paragraph{Prediction} 
We first predict a candidate feedback Nash equilibrium trajectory using the current cost matrices and setting the future unknown cost matrices equal to the known value at time $t+W$. That is, define $\bar{\mathcal{H}}_{t} := (\bar{x}_{1}, (Q_{\tau+1|t})_{\tau=1}^{T-1},(R_{\tau|t}^{1},R_{\tau|t}^{2})_{\tau=1}^{T-1})$,
where $Q_{\tau+1|t},R_{\tau|t}^{1},R_{\tau|t}^{2}$ are defined in \eqref{eq:informationQ} and \eqref{eq:informationR}. Note that the above cost matrices provide a valid structure for the LQ-DFG to be LQ-DFPG based on Assumption \ref{assumption:potentialRepeat}. Moreover, when $T-W < t \leq T-1$, $\bar{\mathcal{H}}_{t} = \mathcal{I}$. This suggests that when enough information is provided, we can find such candidate trajectory that is identical to the feedback Nash equilibrium trajectory obtained with full information $\mathcal{I}$. 
The predicted feedback Nash equilibrium trajectories given the information available at $t$ are
\begin{align}\label{eq:predictionDFL}
    ((x_{\tau|t})_{\tau=1}^{T},(\mathbf{u}_{\tau|t})_{\tau=1}^{T-1}) = \text{DFG}(\bar{\mathcal{H}}_{t},T,A,\contB),
\end{align}
where the operator $\text{DFG}$ defined in \eqref{eq:DFG} that uses $\bar{\mathcal{H}}_{T}$ to output trajectories and controls of the corresponding feedback Nash equilibrium solution from two players.


\paragraph{Tracking} 
Given the predicted feedback Nash equilibrium state and control trajectories \eqref{eq:predictionDFL}, we propose tracking them using a feedback control policy of the form

\begin{align}\label{eq:policy}
\mathbf{u}_{t} = 
\begin{bmatrix}
    \pi_{t}^{1}(x_{t}, \bar{\mathcal{H}}_{t})\\
    \pi_{t}^{2}(x_{t}, \bar{\mathcal{H}}_{t})
\end{bmatrix}=
    \begin{bmatrix}
        [\bar{K}]_{1}\\
        [\bar{K}]_{2}
    \end{bmatrix}(x_{t}-x_{t|t}) + 
    \begin{bmatrix}
        [\mathbf{u}_{t|t}]_{1}\\
        [\mathbf{u}_{t|t}]_{2}
    \end{bmatrix},
\end{align}
where $\bar{K}\in \mathbb{R}^{m\times n}$ is a gain matrix such that $\rho(A+\mathbf{B}\bar{K}) < 1$, $\pi_{t}^{i}(x_{t}, \bar{\mathcal{H}}_{t})$ is the policy for the $i$-th player for $i \in \{1,2\}$, $t \in \mathbb{N}_{T-1}$, and $[\bar{K}]_{i}$ and $[\mathbf{u}_{t|t}]_{i}$ have appropriate dimensions.
In the next section, we show the PoU bounds for this algorithm.


\subsection{\emph{PoU} Analysis}
The following theorem establishes PoU lower and upper bounds for our proposed algorithm.

\begin{theorem}[PoU Bounds]\label{thm:main}
    Consider the linear system \eqref{eq:linsys}, a given time horizon $T \geq 1$, and a preview window length $W\in \mathbb{N}_{T-1}\cup \{0\}$. For time $t\in \mathbb{N}_{T-1}$ and $i\in \{1,2\}$, consider policy \eqref{eq:policy}, under Assumptions \ref{assumption:parameters}--\ref{assumption:potentialRepeat}, the PoU defined by \eqref{eq:regret} satisfies
    \begin{align}\label{eq:mainBound}
        -C_{1}^{'}\gamma^{2W} \!\!-\!C_{2}^{'}\gamma^{W} \!\!-\!C_{3}^{'}\varepsilon_{K} &< \!\text{PoU}_{T}((\Pi_{t})_{t=1}^{T-1})\\
        &\!<\! C_{1}\gamma^{2W}\! +\! C_{2}\gamma^{W}\!+\! C_{3}\varepsilon_{K}\notag,
    \end{align}
    where $\gamma \in (0,1)$ and $\varepsilon_{K}$ is monotonically increasing w.r.t. $\gamma$, $\|\contB\|$, and $\max_{t}|\lambda_{max}([R_{t}^{2}]_{22}-[R_{t}^{1}]_{22})|$. Moreover, constants $C_{1}, C_{2},C_{3}, C_{1}^{'},C_{2}^{'}$ and $C_{3}^{'}$ are monotonically increasing w.r.t. $\lambda_{max}(R_{max}^{'})$ and $\lambda_{max}(Q_{max})$, and the inverse of $\rho(A+\contB\bar{K})$, $\lambda_{min}(R_{min}^{'})$ and $\lambda_{min}(Q_{min})$. 
\end{theorem}
\begin{proof}
 Please see the Appendix.
\end{proof}
\begin{remark}
Following Remark \ref{remark:optimalControl}, in this special case of LQ game we have $C_{2}^{'} = 0$, $C_{3}^{'} = 0$, $\varepsilon_{K} = 0$ and $C_{2} = 0$. Moreover, 
\begin{align*}
    0 \leq \text{PoU}_{T}((\Pi_{t})_{t=1}^{T-1})< C_{1}\gamma^{2W}.
\end{align*}
In this work, $C_{1}\gamma^{2W}$ is a slightly pessimistic upper bound of $\text{PoU}$.
A less pessimistic upper bound can be found by the RHS of dynamic regret bound in \cite[Equation (15)]{chen_regret_2023}. 
\end{remark}

Theorem \ref{thm:main} provides an insight into the relationship between PoU and the preview-window length $W$.
Specifically, we see that for a fixed $\mathcal{I}$, the terms $C_{1}^{'}\gamma^{2W},C_{2}^{'}\gamma^{W},C_{1}\gamma^{2W}$ and $C_{2}\gamma^{W}$ in the PoU bound decays exponentially fast as the preview-window length $W$ increases. Moreover, when $W$ is sufficiently large, the PoU lower and upper bounds are dominated by $-C_{3}^{'}\varepsilon_{K}$ and $C_{3}\varepsilon_{K}$, respectively.

We shall illustrate the exponential decay of the absolute value of PoU with preview-window length next in simulations.

Before presenting our simulation result, we note that a slightly modified PoU notion
termed \emph{log-relative PoU} can be used to describe the relative error between the decisions made from the algorithm and the (value of the) feedback Nash equilibrium
solution.  Specifically, for any sequence of 
policies $(\Pi_{t})_{t=1}^{T-1}$, define the
\emph{log relative PoU} as \begin{align*}
    \text{logRelPoU}((\Pi_{t})_{t=1}^{T-1}):= \log \bigg|\frac{\text{PoU}_{T}((\Pi_{t})_{t=1}^{T-1})}{\frac{1}{2}\sum_{i=1}^{2} J_{i,T}(\bar{x}_{1},(\Pi_{t}^{*})_{t=1}^{T-1})}\bigg|,
\end{align*}
where policy $(\Pi_{t}^{*})_{t=1}^{T-1}$ generates the feedback Nash equilibrium trajectory as defined in \eqref{eq:nashIneq}.
In certain scenarios, the cost incurred by the actual feedback Nash
equilibrium may be large due to the selection of system and cost matrices. Thus,
logarithmic relative PoU can be a more practical measure of the quality of
decisions when we are interested in numerical results.
\section{NUMERICAL SIMULATIONS}\label{sec:numerical}
In this section, we numerically illustrate the performance of our proposed algorithm. 
We consider a slightly modified demand-side management example from the one studied in \cite{wu_demand_2011}. 
In this scenario, two users aim to schedule their electricity usage over a period of $T$. 
Their objective is to minimise their usage costs while maintaining a solar-powered community battery at a nominal level at time $t$. 
The battery level evolves as a function of the current drawn $u_{t}^{i}$ by each user $i \in \{1,2\}$ at time $t$ and a constant recharge rate according to 
\begin{equation*}
    x_{t+1} = 
    \begin{bmatrix}
        a & 0\\
        0 & 0.9
    \end{bmatrix}
    x_{t} + 
    \begin{bmatrix}
        -b_{1} & -b_{2}\\
        0 & 0
    \end{bmatrix}\mathbf{u}_{t},
\end{equation*}
where $a \geq 1$ is the constant recharge rate of the battery (with current drawn from either renewable or nonrenewable sources) and $x_{t}$ is the deviation of the community battery level from its nominal value.
The cost matrices $Q_{t}$ and $R_{t}^{i}$ penalise the deviation of the community battery level from its nominal value and usage costs, respectively. 
Users have knowledge of future cost matrices only over a preview window of length $W \in \mathbb{N}_{T-1}\cup \{0\}$ due to variable environmental conditions affecting battery-recharge costs and the power availability in the battery.



For the purpose of simulations, we consider the following example by adopting the sufficient condition of LQ-DFG that satisfies our assumptions from Section \ref{sec:assumptions}. Consider $a=1.6,b_{1}=0.85$ and $b_{2} =0.89$.
The preview window $W$ varies from $0$ to $6$, and $T$ ranges from $1$ to $40$. The cost matrices are
\begin{equation*}
    Q_{t} = 
    \begin{bmatrix}
        l_{t} & -d_{t}\\
        -d_{t} & 0
    \end{bmatrix},\
    R_{t}^{1} = 
    \begin{bmatrix}
        r_{1,t} & 0\\
        0 & 0
    \end{bmatrix},\
    R_{t}^{2} = 
    \begin{bmatrix}
        0 & 0\\
        0 & r_{2,t}
    \end{bmatrix},
\end{equation*}
where $\frac{b_{1}^{2}}{r_{1,t}}=\frac{b_{2}^{2}}{r_{2,t}} = \beta_{t}$. The above parameters satisfy Assumption \ref{assumption:parameters}-\ref{assumption:potentialRepeat}, and the proof can be found in the Appendix. We consider cost matrices drawn randomly according to $\beta_{t} \sim \mathcal{U}[10,110]$, $l_{t} \sim \mathcal{U}[10,110]$ and $d_{t}\sim \mathcal{U}[-110,-10]$, where $\mathcal{U}[\cdot,\cdot]$ denotes the uniform distribution. 

The plot in Figure \ref{fig:1relPreview} reports the average of \emph{log relative PoU} versus time horizon $T$ from $T = 1$ to $35$ with preview window length fixed at $W=1$. From Figure \ref{fig:1relPreview}, we see that the (average) \emph{log relative PoU} grows from $T = 1$ to $T = 15$, and saturates at an approximate value of $-4$ at $T \geq 15$. The saturation of the average \emph{log relative PoU} is consistent with the independence of time horizon $T$ and PoU lower and upper bounds observed in Theorem \ref{thm:main} for sufficiently large $T$. 

The plot in Figure \ref{fig:20relTimehorizon} reports the average of \emph{log relative PoU} versus preview window length $W$ from $W = 0$ to $6$ with time horizon fixed at $T=20$. From Figure \ref{fig:20relTimehorizon}, we see that the (average) \emph{log relative PoU} decays at a faster-than-exponential rate from $W =0$ to $6$. The decay of the \emph{log relative PoU} is consistent with the relation between preview window length $W$ and the PoU bounds established in Theorem \ref{thm:main}.

\begin{figure}
        \label{fig:experiments}
     \subfigure[logRelPoU vs. Time Horizon in $W=1$]{
         \includegraphics[width=.45\textwidth]{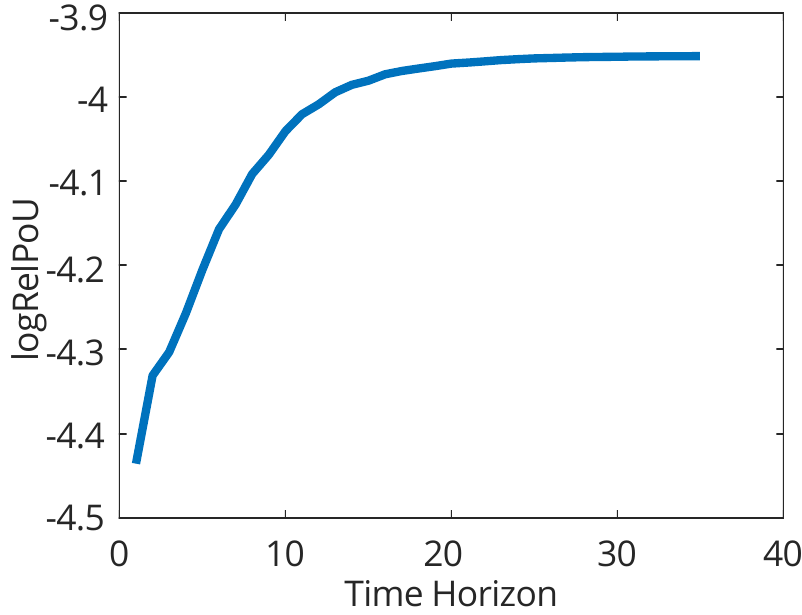}
    \label{fig:1relPreview}}
     \subfigure[logRelPoU vs. Preview Window Length in $T=20$]{
         \includegraphics[width=.45\textwidth]{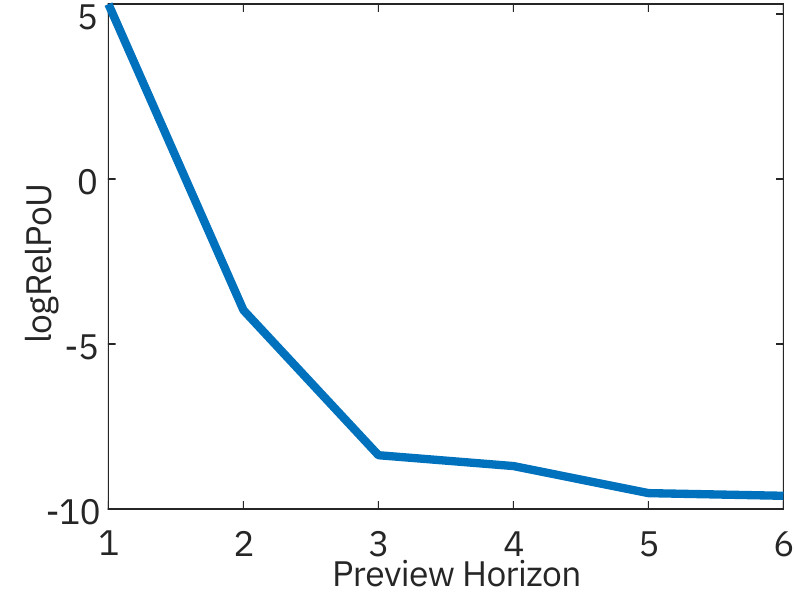}
    \label{fig:20relTimehorizon}}
    \caption{Performance measure $\text{logRelPoU}_{T}$ for simulated systems.}
\end{figure}
\section{Conclusion and Future Work}\label{sec:conclusions}
This paper presents a two-player dynamic feedback potential game problem with sequentially revealed costs and proposes a novel algorithm for solving it. The notion of the price of uncertainty is introduced to derive PoU lower and upper bounds, quantifying the performance of the proposed algorithm. 
Future work will extend the analysis to $N$-players, and establish the connections between PoU and the $\varepsilon$-Nash equilibrium.
\bibliographystyle{plain}
\bibliography{References}
\appendix
Before stating the proof of Theorem~\ref{thm:main}, in the next subsection, we introduce the necessary lemmas and propositions. 
\section{Auxiliary Lemmas}
\begin{lemma}[Theorem 6, \cite{prasad_structure_2023}]\label{lemma:parametersPotentialGame}
    Suppose parameters $A$, $\mathbf{B}$, $Q_{t+1}$, $R_{t}^{i}$ satisfy Assumption \ref{assumption:parameters} for $i\in \{1,2\}$ and $t\in \mathbb{N}_{T-1}$, 
    then, the LQ-DFG as defined by the parameters in \eqref{eq:DFG} constitutes an LQ-DFPG as described in Definition \ref{def:LQ-DFPG}.
\end{lemma}

\begin{lemma}[Theorem 5, \cite{prasad_structure_2023}]\label{lemma:gamePGrelation}
    Consider an LQ-DFPG and an LQ-OCP described by Definitions \ref{def:LQOCP} and Section \ref{section:LQ-DFG}, respectively. Under Assumption \ref{assumption:bounds}, the feedback Nash equilibrium for the LQ-DFPG defined in Definition \ref{eq:DFG} by parameters $\{\bar{x}_{1},\{Q_{t}\}_{t=1}^{T},\{R_{t}^{i}\}_{i=1,t=1}^{2,T-1}\}$, is identical to the solutions of LQ-OCP defined in Definition \ref{def:LQOCP}, by parameters $\{\bar{x}_{1},\{\bar{Q}_{t}\}_{t=1}^{T},\{\bar{R}_{t}\}_{t=1}^{T-1}\}$, if
    \begin{align*}
        &\bar{P}_{T} = \bar{Q}_{T}, \bar{\Theta}_{t} = \bar{R}_{t} + \contB^{\transpose}\bar{P}_{t+1}\contB, \bar{K}_{t} = \bar{\Theta}_{t}^{-1}\contB^{\transpose}\bar{P}_{t+1}A,\\
        &\bar{P}_{t} = \bar{Q}_{t} + \bar{K}_{t}^{\transpose}\bar{R}_{t}\bar{K}_{t} + (A+\contB\bar{K}_{t})^{\transpose}\bar{P}_{t+1}(A+\contB\bar{K}_{t}),\\
        &[R_{t}^{i}]_{ii} + B^{i\transpose}P_{t+1}^{i}B^{i} = [\bar{R}_{t}]_{ii} + B^{i\transpose}\bar{P}_{t+1}B^{i},\\
        &B^{i\transpose}P_{t+1}^{i}A = B^{i\transpose}\bar{P}_{t+1}A, \\
        &\bar{\Theta}_{t} \in \mathbb{S}_{++}^{2m},\\
        &[R_{t}^{i}]_{ij} + B^{i\transpose}P_{t+1}^{i}B^{j} = [\bar{R}_{t}]_{ij} + B^{i\transpose}\bar{P}_{t+1}B^{j},i\neq j,\\
        &[R_{t}^{1}]_{12} + B^{1\transpose}P_{t+1}^{1}B^{2} = ([R^{2}_{t}]_{21} + B^{2\transpose}P^{2}_{t+1}B^{1})^{\transpose},
    \end{align*}
    for $i,j\in \{1,2\}$ and $t \in \mathbb{N}_{T-1}$. 
\end{lemma}
\begin{lemma}[Theorem 7, \cite{prasad_structure_2023}]\label{lemma:prasadThm7}
    For a given integer $T\geq 1$, define
    \begin{align*}
        \Omega := ((\bar{Q}_{1},\bar{Q}_{2},\cdots,\bar{Q}_{T},\bar{R}_{1},\cdots,\bar{R}_{T-1})|\mathcal{K}),
    \end{align*}
    where $\mathcal{K}$ is the set of conditions that, for time $t \in \mathbb{N}_{T}$, $\bar{P}_{t}$ is such that
            $\mathbf{B}^{\transpose}\bar{P}_{t}A = \mathbf{B}^{\transpose}P_{t}^{i}A$ for $i \in \{1,2\}$.
       Define $\bar{R}_{t}$ as
            \begin{align}\label{eq:matrixR}
            \bar{R}_{t} := \Theta_{t} - \mathbf{B}^{\transpose}\bar{P}_{t+1}\mathbf{B},
            \end{align}
        where $\Theta_{t}$ is defined in \eqref{eq:Theta},
        with $\bar{Q}_{t}$ such that
            $\bar{Q}_{t} = Q_{t} + K_{t}^{\transpose}(R_{t}^{1}-\bar{R}_{t})K_{t}$,
        and $K_{t}$ is defined in \eqref{eq:gainK}.
    The set $\Omega$ is non-empty, and every element in $\Omega$ leads to an LQ-OCP with a solution identical to the solution of an LQ-DFPG.
\end{lemma}
\begin{remark}\label{remark:rp}
    Matrix $\bar{R}_{t}$ defined in \eqref{eq:matrixR} satisfies $\bar{R}_{t} = R_{t}^{p}$.
\end{remark}
\begin{corollary}\label{corollary:boundedR}
    Under Assumptions \ref{assumption:controllable} and \ref{assumption:bounds}, $\bar{R}_{t}$ defined in \eqref{eq:matrixR} satisfies $0 \prec R_{min} \preceq \bar{R}_{t} \preceq R_{max}$,
    for $t \in \mathbb{N}_{T-1}$.
\end{corollary}

To help us present preceding remarks and lemmas, and similar to the operator \eqref{eq:predictionDFL} for finding the feedback Nash equilibrium trajectories of the game defined by the matrices $A, \mathbf{B},(Q_{\tau+1},R_{1,\tau},R_{2,\tau})_{\tau=1}^{T-1}$. With slightly abuse of notations, we define the following operators associated with the computation of parameters in Assumption \ref{assumption:parameters}:
    \begin{align}
        &K_{t} := \mathrm{K}((Q_{\tau+1},R_{\tau}^{1},R_{\tau}^{2})_{\tau=t}^{T-1}, A, \mathbf{B}),\label{eq:op2}\\
        &\bar{R}_{t} := \bar{\mathrm{R}}((Q_{\tau+1},R_{\tau}^{1},R_{\tau}^{2})_{\tau=t}^{T-1}, A, \mathbf{B}).\label{eq:op3}
    \end{align}
Operators $\mathrm{K}$ and $\mathrm{\bar{R}}$ involve solving coupled Riccati difference equations defined in \eqref{eq:pIteration} for computing $P_{t}^{i}$ for $t\in \mathbb{N}_{T}$ and $i \in \{1,2\}$, with the given input systems and cost parameters. Then return results are based on different formulas depending on the arguments and $P_{t}^{i}$.

\begin{remark}
    For $1\leq \tau \leq t \leq T-1$, by Assumption \ref{assumption:parameters} and Assumption \ref{assumption:potentialRepeat}, , let $\bar{R}_{\tau}$ from \eqref{eq:matrixR} and $\bar{R}_{\tau|t} = \bar{\mathrm{R}}((Q_{k+1|t},R_{k|t}^{1},R_{k|t}^{2})_{k=\tau}^{T-1}, A, \mathbf{B})$, these satisfy
        $\bar{R}_{\tau|t} = \bar{R}_{\tau}$.
\end{remark}
\begin{lemma}\label{lemma:matrixK}
    For positive integers $n,m$ such as $m \leq n$. Consider any $R\in \mathbb{S}^{m}_{++}$, $P\in \mathbb{S}^{n}_{++}$ and $F \in \mathbb{R}^{n\times m}$. Let $K = (R+F^{\transpose}PF)^{-1}F^{\transpose}P$, then $\|K\| \leq \frac{1}{\sigma_{min}^{+}(F)}$.
\end{lemma}
\begin{proof}
    Let $\sigma_{R} := \sigma_{min}(R)$, observe $\lambda_{max}(K^{\transpose}K)$:
\begin{align*}
    &\lambda_{max}(PF(R+F^{\transpose}PF)^{-1}(R+F^{\transpose}PF)^{-1}F^{\transpose}P)\\
    &\stackrel{\text{i}}{=} \lambda_{max}((R+F^{\transpose}PF)^{-1}F^{\transpose}P^{2}F(R+F^{\transpose}PF)^{-1})\\
    &\stackrel{ii}{\leq} \lambda_{max}((\sigma_{R}I+F^{\transpose}PF)^{-1}F^{\transpose}P^{2}F(\sigma_{R}I+F^{\transpose}PF)^{-1}),
\end{align*}
where step (i) uses the property that the maximum eigenvalue of $H^{\transpose}H$ is the same as the maximum eigenvalue of $HH^{\transpose}$, and step (ii) uses $R \succeq \sigma_{R}I$.
Fy SVD, $F = U_{F}\Sigma_{F}V_{F}^{\transpose}$, then
\begin{align*}
    &\lambda_{max}((\sigma_{R}I+F^{\transpose}PF)^{-1}F^{\transpose}P^{2}F(\sigma_{R}I+F^{\transpose}PF)^{-1})\\
    &=\lambda_{max}((\sigma_{R}I+\Sigma_{F}^{\transpose}\bar{P}\Sigma_{F})^{-1}\Sigma_{F}^{\transpose}\bar{P}^{2}\Sigma_{F}(\sigma_{R}I+\Sigma_{F}^{\transpose}\bar{P}\Sigma_{F})^{-1}),
\end{align*}
where $\bar{P} = U_{F}^{\transpose}PU_{F}$.
Matrix $\Sigma_F$ can be written as
\begin{align*}
    \Sigma_{F} = 
    \begin{bmatrix}
        \Sigma_{Fr} & \mathbf{0}_{r,m-r}\\
        \mathbf{0}_{n-r,r} & \mathbf{0}_{n-r,m-r}
    \end{bmatrix},
\end{align*}
where $\text{rank}(F) = r$, $1\leq r \leq \min(n,m)$, and $\Sigma_{Fr} = \text{diag}(\sigma_{1}(F),\cdots, \sigma_{r}(F))$. 
Therefore,
\begin{align*}
    &\lambda_{max}((\sigma_{R}I+\Sigma_{F}^{\transpose}\bar{P}\Sigma_{F})^{-1}\Sigma_{F}^{\transpose}\bar{P}^{2}\Sigma_{F}(\sigma_{R}I+\Sigma_{F}^{\transpose}\bar{P}\Sigma_{F})^{-1})\\
    &=\lambda_{max}(
    \begin{bmatrix}
        \sigma_{R}I_{r\times r} \!\!+\!\! \Sigma_{F}[\bar{P}]_{11}\Sigma_{F} & \mathbf{0}_{r\times m-r}\\
        \mathbf{0}_{m-r\times r} & \sigma_{R}I_{m-r\times m-r}
    \end{bmatrix}^{-1}\\
    &\qquad\qquad\begin{bmatrix}
        \Sigma_{F}[\bar{P}^{2}]_{11}\Sigma_{F} & \mathbf{0}_{r\times m-r}\\
        \mathbf{0}_{m-r\times r} & \mathbf{0}_{m-r\times m-r}
    \end{bmatrix}\\
   &\qquad \qquad \begin{bmatrix}
        \sigma_{R}I_{r\times r} \!\!+\!\! \Sigma_{F}[\bar{P}]_{11}\Sigma_{F} & \mathbf{0}_{r\times m-r}\\
        \mathbf{0}_{m-r\times r} & \sigma_{R}I_{m-r\times m-r}
    \end{bmatrix}^{-1})\\
    &= \lambda_{max}((\sigma_{R}I + \Sigma_{F}[\bar{P}]_{11}\Sigma_{F})^{-1}\Sigma_{F}[\bar{P}^{2}]_{11}\Sigma_{F}(\sigma_{R}I + \Sigma_{F}[\bar{P}]_{11}\Sigma_{F})^{-1})\\
    &\stackrel{\text{i}}{=} \lambda_{max}( (\sigma_{R}\Sigma_{F}^{-1}[\bar{P}]_{11}^{-1}+ \Sigma_{F})^{-1}(\sigma_{R}[\bar{P}]_{11}^{-1}\Sigma_{F}^{-1}+ \Sigma_{F})^{-1} )\\
    &= \lambda_{max}(\Sigma_{F} (\sigma_{R}[\bar{P}]_{11}^{-1}+ \Sigma_{F}^{2})^{-2}\Sigma_{F})\\
    &\leq \lambda_{max}(\Sigma_{F}^{-2})= \frac{1}{(\sigma_{min}^{+}(F))^{2}}.
\end{align*}
Step i) is due to $\bar{P}$ and $\Sigma_{F}$ are invertible.
Therefore, $\|K\| \leq \frac{1}{\sigma_{min}^{+}(F)}$.
\end{proof}

\begin{corollary}\label{corrolary:boundedK}
    Suppose the cost matrices $(Q_{t})_{t=1}^{T}$ and $(R_{t}^{i})_{t=1,i=1}^{T-1,2}$ satisfy the conditions described in \eqref{eq:costFPDG1}, \eqref{eq:costFPDG3}, and Assumption \ref{assumption:bounds}.
    Consider operator $\mathrm{K}$ defined in \eqref{eq:op2}. For a given preview horizon $W\in \mathbb{N}_{T-1}$ and all $\tau \in \mathbb{N}_{T}$, the solution $\mathbf{u}_{\tau|t}$ defined in \eqref{eq:predictionDFL} is given by $\mathbf{u}_{\tau|t} = K_{\tau|t}x_{\tau|t}$, where $K_{\tau|t} = \mathrm{K}((Q_{k+1|t},R_{k|t}^{1},R_{k|t}^{2})_{k=\tau}^{T-1}, A, \mathbf{B})$, and $\|K_{\tau|t}\| \leq \frac{\sigma_{max}(A)}{\sigma_{min}^{+}(\contB)}$.
\end{corollary}
\begin{corollary}\label{corollary:positiveQ}
    For any given $\tau,t\in\mathbb{N}_{T-1}$, 
       $ \bar{Q}_{\tau|t} = Q_{\tau|t} + K_{\tau|t}^{\transpose}(R_{\tau|t}^{1} - \bar{R}_{\tau|t})K_{\tau|t} \in \mathbb{S}^{n}_{++}$.
\end{corollary}

\begin{remark}
    Assumption \ref{assumption:lowerQ} suggests that, if $\lambda_{max}(\bar{R}_{t}-R_{t}^{1})$ is positive, then matrices $Q_{\tau|t}$ are positive definite. By Remark \ref{remark:rp}, $\bar{R}_{t} = R_{t}^{p}$, therefore we need $\lambda_{max}(R^{p}_{t}-R_{t}^{1})$ to be positive. This will help us to develop the contraction between $K_{\tau|t}$ and $K_{\tau|t_{0}}$ characterised by metric $\delta_{\infty}(\cdot,\cdot)$ for $\tau \in \mathbb{N}_t$ and $t_{0} \in \mathbb{N}_{T-1}$.
\end{remark}
\begin{lemma}\label{lemma:boundedP}
Consider $A,\mathbf{B}$ from \eqref{eq:linsys}, $\bar{R}_{t}$ and $\bar{Q}_{t}$ from Lemma \ref{lemma:prasadThm7} that satisfy Assumption \ref{assumption:bounds}. For all $t \in \mathbb{N}_{T-1}$, let $\bar{P}_{T} = \bar{Q}_{T}$ and 
    \begin{align*}
        \bar{K}_{t}& = -(\bar{R}_{t}+ \mathbf{B}^{\transpose}\bar{P}_{t+1}\mathbf{B})^{-1}\mathbf{B}^{\transpose}\bar{P}_{t+1}A,\\
        \bar{P}_{t}& = \bar{Q}_{t} + \bar{K}_{t}^{\transpose}\bar{R}_{t}\bar{K}_{t} + (A+\mathbf{B}\bar{K}_{t})^{\transpose}\bar{P}_{t+1}(A+\mathbf{B}\bar{K}_{t}),
    \end{align*}
    for $t \in \mathbb{N}_{T-1}$. There exist positive definite matrices $\bar{P}_{min}$ and $\bar{P}_{max}$ such that $
        \bar{P}_{min} \preceq \bar{P}_{t} \preceq \bar{P}_{max}$.
\end{lemma}
\begin{proof}
    By Corollary \ref{corrolary:boundedK} and Assumption \ref{assumption:lowerQ}, there exist positive definite matrices $\bar{Q}_{min},\bar{Q}_{max}$ such that $
        \bar{Q}_{min} \preceq \bar{Q}_{t} \preceq \bar{Q}_{max}$,
    for $t\in\mathbb{N}_{T}$.
    This, together with the Assumption \ref{assumption:bounds}, and repeat procedure of \cite[Appendix D, Proposition 11]{zhang_regret_2021} with $Q_{min}$, $Q_{max}$, $R_{max}$, $Q_{t}^{max}$, $R_{t}^{max}$ and $P_{max}$ replaced by $\bar{Q}_{min}$, $\bar{Q}_{max}$, $\bar{R}_{max}$, $\bar{Q}_{max}$, $\bar{R}_{max}$ and $\bar{P}_{max}$, respectively, complete the proof.
\end{proof}

\begin{lemma}\label{lemma:m}
    For $T,S,V_{1},V_{2}\in \mathbb{S}_{++}^{n}$, if
$m \geq 1+\frac{\lambda_{max}(V_{1}-V_{2})}{\lambda_{min}(T+V_{2})}$,
    then for any non-zero vector $x\in R^{n}$, we have
    \begin{align}
        \frac{\quadinner{T+V_{1}}}{\quadinner{S+V_{2}}} \leq m\frac{\quadinner{T+V_{2}}}{\quadinner{S+V_{2}}}.
    \end{align}
\end{lemma}
\begin{proof}
    For any nonzero vector $x\in R^{n}$, we have
    \begin{align*}
        m \geq 1+\frac{\lambda_{max}(V_{1}-V_{2})}{\lambda_{min}(T+V_{2})}
        \geq 1+ \frac{\quadinner{V_{1}-V_{2}}}{\quadinner{T+V_{2}}}
        = \frac{\quadinner{T+V_{1}}}{\quadinner{T+V_{2}}}.
     \end{align*}
     Due to $\quadinner{S+V_{2}} > 0$ and the fact that $T,V_{2},V_{1}$ are positive definite, we have that 
         $m\frac{\quadinner{T+V_{2}}}{\quadinner{S+V_{2}}} \geq \frac{\quadinner{T+V_{1}}}{\quadinner{S+V_{2}}}$.
\end{proof}

\begin{definition}
    For any $X,Y\in \mathbb{S}_{++}^{n}$, define the operator $\delta_{\infty}(\cdot, \cdot)$ as
  $ \delta_{\infty}(X, Y) := \| \log(Y^{-\frac{1}{2}}XY^{-\frac{1}{2}})\|_{\infty}$ \cite[Section 2]{thompson_certain_1963}, where $\|\cdot\|_{\infty}$ denotes the matrix infinity-norm.
\end{definition}

\begin{proposition}\label{proposition:deltaRatio}
    For any $X,Y\in \mathbb{S}_{++}^{n}$,
$ \delta_{\infty}(X,Y) = \max(\log(\sup_{\xi\neq0} \frac{\xi^{\transpose}X\xi}{\xi^{\transpose}Y\xi}),\log(\sup_{\xi\neq0} \frac{\xi^{\transpose}Y\xi}{\xi^{\transpose}X\xi}))$.
\end{proposition}

\begin{remark}\label{remark:delta}
    For $T,S,V_{1},V_{2}\in \mathbb{S}_{++}^{n}$, without loss of generality, assume that
    $\sup_{x\neq 0} \frac{\quadinner{T+V_{1}}}{\quadinner{S+V_{2}}} \geq 1$.
   Suppose a positive scalar $m$ satisfies $m \geq 1+\frac{\lambda_{max}(V_{1}-V_{2})}{\lambda_{min}(T+V_{2})}$.  Based on Lemma \ref{lemma:boundedP}, \ref{lemma:m}, and Proposition \ref{proposition:deltaRatio}, we have $\delta_{\infty}(T+V_{1},S+V_{2}) \leq \log(m) + \gamma\delta_{\infty}(T,S)$,
    where $0<\gamma<1$.
\end{remark}

Before we present our next lemma that associates the bounds of  $\|\bar{P}_{\tau|t}-\bar{P}_{\tau|t_{0}}\|$ and $\|\bar{K}_{\tau|t}-\bar{K}_{\tau|t_{0}}\|$ for $1\leq \tau\leq t\leq t_{0}\leq T$, we define the following constants:
\begin{align}
    &\alpha_{\tau,t} := \lambda_{max}(A^{\transpose}(\bar{P}_{\tau+1|t}^{-1}+B\bar{R}_{\tau|t}^{-1}B^{\transpose})^{-1}A),\\
    &\beta_{\tau,t} := \lambda_{min}(\bar{Q}_{\tau|t}),\gamma := \max_{1\leq \tau,t \leq T} \frac{\alpha_{\tau,t}}{\alpha_{\tau,t}+\beta_{\tau,t}},\label{eq:gamma}\\
    &h := \max_{(\tau,t_{0}| \tau,\leq t_{0})} \delta_{\infty}(\bar{P}_{\tau|t},\bar{P}_{\tau|t_{0}}),\\
    &C_{P}=\frac{C_{P1}\lambda_{max}(\bar{P}_{\tau|t_{0}})(\exp(h)-1)}{h},\label{eq:CP}\\
    &\bar{\omega} = \max_{\tau,t,t_{0}}\lambda_{max}(A^{\transpose}[(\bar{P}_{\tau+1|t}^{-1}+B\bar{R}_{\tau|t}^{-1}B^{\transpose})^{-1}\notag\\
    &\qquad \qquad \qquad-(\bar{P}_{\tau+1|t_{0}}^{-1}+B\bar{R}_{\tau|t_{0}}^{-1}B^{\transpose})^{-1}]A),\notag\\ 
    &\varepsilon_{1} := \log(1+\frac{\bar{\omega}}{\lambda_{min}(\bar{Q})}),\label{eq:epsilon1}\\
    &\varepsilon_{P}=\frac{\varepsilon_{1}\lambda_{max}(\bar{P}_{\tau|t_{0}})(\exp(h)-1)}{h(1-\gamma)},\label{eq:epsilonP}\\
    &G_{max} := \max_{t,\tau} \|(\bar{R}_{\tau}+\contB^{\transpose}\bar{P}_{\tau+1|t}\contB)^{-1}\|,\\
    &C_{K}^{'}= G_{max}^{2}\|R_{max}\contB^{\transpose}\|C_{P},\label{eq:CK}\\
    &\varepsilon_{K}^{'}= G_{max}^{2}\|R_{max}\contB^{\transpose}\|\varepsilon_{P}\label{eq:epsilonK}.
\end{align}

\begin{lemma}\label{lemma:boundedPK}
     For $1\leq \tau\leq t\leq t_{0}\leq T$, 
    suppose $W\in \mathbb{N}_{T-1}$ is the preview window length, consider scalars $\varepsilon_{P},\varepsilon_{K},C_{P},C_{K}^{'}$ and $\gamma \in (0,1)$ defined in \eqref{eq:gamma},\eqref{eq:CP},\eqref{eq:epsilonP},\eqref{eq:CK} and \eqref{eq:epsilonK}, respectively. The distance between $\bar{P}_{\tau|t}$ and $\bar{P}_{\tau|t_{0}}$ satisfies $\|\bar{P}_{\tau|t}-\bar{P}_{\tau|t_{0}}\| \leq C_{P}\gamma^{t-\tau+W}+\varepsilon_{P}$, and the distance between $\bar{K}_{\tau|t}$ and $\bar{K}_{\tau|t_{0}}$ satisfies $\|\bar{K}_{\tau|t}-\bar{K}_{\tau|t_{0}}\| \leq C_{K}^{'}\gamma^{t-\tau+1+W}+\varepsilon_{K}^{'}$.
\end{lemma}
\begin{proof}
     For $1\leq \tau\leq t\leq t_{0}\leq T$, by \cite[Lemma D.2]{krauth_finite-time_2019}, Lemmas \ref{lemma:boundedP} \& \ref{lemma:m}, and Remark \ref{remark:delta}, we have $\delta_{\infty}(\bar{P}_{\tau|t},\bar{P}_{\tau|t_{0}}) < C_{P_{1}}\gamma^{t-\tau+W}+\frac{\varepsilon_{1}}{1-\gamma}$,
    where $\varepsilon_{1}$ is defined in \eqref{eq:epsilon1}. By monotonicity of function $\frac{e^x-1}{x}$ for $x > 0$, we have $\frac{\exp(\delta_{\infty}(\bar{P}_{\tau|t},\bar{P}_{\tau|t_{0}}))-1}{\delta_{\infty}(\bar{P}_{\tau|t},\bar{P}_{\tau|t_{0}})} \leq \frac{\exp(h)-1}{h}$.
    Thus, $\|\bar{P}_{\tau|t}-\bar{P}_{\tau|t_{0}}\| < C_{P}\gamma^{W}+\varepsilon_{P}$. By repeating procedures as the prove in \cite[Lemma 8, (20) and (21)]{chen_regret_2022} by replacing $B,R_{\tau},P_{\tau+1|t}$ and $P_{\tau+1|t_{0}}$ from \cite[Lemma 8, (20) and (21)]{chen_regret_2022} to $\contB,\bar{R}_{\tau},\bar{P}_{\tau+1|t}$ and $P_{\tau+1|t_{0}}$, we have $\|\bar{K}_{\tau|t}-\bar{K}_{\tau|t_{0}}\| < C_{K}^{'}\gamma^{t-\tau-1+W}+\varepsilon_{K}^{'}$.
\end{proof}

\begin{remark}
    When $\tau = t$ and $t_{0} = T$, we have $\|\bar{K}_{t|t}-\bar{K}_{t}\| < C_{K}^{'}\gamma^{W-1} + \varepsilon_{K}^{'}.$
\end{remark}

\begin{lemma}\label{lemma:multGain}
For integers $1\leq t_{0} \leq t_{1}\leq t\leq T-1$, consider $A$, $\mathbf{B}$ from \eqref{eq:linsys} and $\bar{K}_{\tau|t} = \mathrm{K}((Q_{k+1|t},R_{k|t}^{1},R_{k|t}^{2})_{k=\tau}^{T-1}, A, \mathbf{B})$ where $\mathrm{K}$ is defined in \eqref{eq:op2}, then there exist positive scalars $C_{fb}$ and $0 < \eta < 1$ such that $\left \| \prod_{\tau=t_{0}}^{t_{1}}(A+\mathbf{B}\bar{K}_{\tau|t})  \right\| \leq C_{fb}\eta^{t_{1}-t_{0}+1}$.
\end{lemma}
\begin{lemma}\label{lemma:distanceXt}
    For any $T \geq 1$, $W\in \mathbb{N}_{T-1}\cup \{0\}$ and $t \in \mathbb{N}_{T}$, consider state $x_{t}$ that generated by \eqref{eq:policy} and the state $x_{t|t}$ as an element of the solution from \eqref{eq:predictionDFL}. Let $\varepsilon_{K} := \|\contB\|\varepsilon_{K}^{'}$ and $C_{K} = \|\contB\|C_{K}^{'}$, where $C_{K}^{'}$ and $\varepsilon_{K}^{'}$ are as in Lemma \ref{lemma:boundedPK}. Then, for any fixed $0 < \varepsilon < 1-\rho(A+\mathbf{B}\bar{K})$, let $q = \rho(A+\mathbf{B}\bar{K}) + \varepsilon$ and $C_{fb} = \sup_{k\geq 1} \frac{\|(A+\mathbf{B}\bar{K})\|^{k}}{q^k}$, the distance between state $x_{t}$ and $x_{t|t}$ satisfies
    \begin{align*}
        \|x_{t}-x_{t|t}\| \leq &C_{fb}^{2}\|\bar{x}_{1}\|q^{t} \bigg [\!\frac{C_{K}\gamma^{W}}{\gamma-1}\bigg(\frac{1-(\frac{\eta\gamma}{q})^{t}}{1-\frac{\eta\gamma}{q}} - \frac{1-(\frac{\eta}{q})^{t}}{1-\frac{\eta}{q}} \bigg) \\
        &+\varepsilon_{K}\bigg(\frac{(t-1)(\frac{\eta}{q})^{t+1}-t(\frac{\eta}{q})^{t}+\frac{\eta}{q}}{(1-\frac{\eta}{q})^{2}}\bigg)\!\! \bigg]
    \end{align*}
    where $\gamma$ is given in Lemma~\ref{lemma:boundedPK} and $\eta$ is as in Lemma~\ref{lemma:multGain}.
\end{lemma}
\begin{proof}
    Suppose $N$ is a positive integer. For any given matrix sequence $(a_{i})_{i=1}^{N}$. Let $I$ be the identity matrix. For $1\leq p_{1}\leq N$ and $1 \leq p_{2} \leq N$, with slight abuse of notation, define the product operator
\begin{equation*}
    \prod_{j=p_{1}}^{p_{2}} a_{j} := 
    \begin{cases}
        a_{p_{2}}a_{p_{2}-1}\dots a_{p_{1}} & \text{if $p_{1} < p_{2}$}\\
        a_{p_{2}} & \text{if $p_{1} = p_{2}$}\\
        I & \text{if $p_{1} > p_{2}$},
    \end{cases}
\end{equation*}
Define $\omega_{t} := x_{t}-x_{t|t}$, $\theta_{\tau| p_{1},p_{2}} := x_{\tau|p_{1}}-x_{\tau|p_{2}}$, where $1 \leq p_{1}\leq p_{2}\leq T$ and $\tau, t\in \mathbb{N}_{T}$. Since $x_{1} = x_{1|t} = x_{1|p_{1}} = x_{1|p_{2}} = \bar{x}_{1}$, consequently, $w_{1} = 0$ and $\theta_{1|p_{1},p_{2}}=0$. We now investigate the dynamics of $w_{t}$ and $\theta_{\tau|p_{1},p_{2}}$. 
For $t\in\mathbb{N}_T$ and $\bar{K}$ defined in \eqref{eq:policy} satisfies Assumption \ref{assumption:controllable}, we have $\omega_{t} = \sum_{i=1}^{t} (A+\mathbf{B}\bar{K})^{t-i} \theta_{i|i-1,i}$.
We now investigate the dynamics of $\theta_{\tau|p_{1},p_{2}}$. Note that $\theta_{1|p_{1},p_{2}} = 0$, and $\theta_{\tau+1|p_{1},p_{2}} = (A+\BK{\tau|p_{1}})\theta_{\tau|p_{1},p_{2}} + \mathbf{B}(\bar{K}_{\tau|p_{1}}-\bar{K}_{\tau|p_{2}})x_{\tau|p_{2}}$.
This implies that 
\begin{align*}
    \theta_{\tau|p_{1},p_{2}}= \sum_{i=1}^{\tau-1}\!&\bigg(\!\prod_{j=i+1}^{\tau-1}(A+\mathbf{B}\bar{K}_{j|p_{1}})\!\!\bigg)\!\!\bigg[\mathbf{B}^{\mathsf{T}}(\bar{K}_{j|p_{1}}^{\mathsf{T}}-\bar{K}_{j|p_{2}}^{\mathsf{T}})\bigg]\\
    &\bigg(\!\prod_{n=1}^{i-1}(A+\mathbf{B}\bar{K}_{n|p_{2}})\!\!\bigg)\bar{x}_{1}.
\end{align*}

By Lemma \ref{lemma:multGain},
$
    \| \prod_{m=n+1}^{\tau}(A+\BK{m|p_{1}})\| \leq C_{fb}\eta^{\tau-n}.
$
By Lemma \ref{lemma:boundedPK}, let $C_{K} := \|\mathbf{B}\|C_{K}^{'}$ and $\varepsilon_{K} := \|\mathbf{B}\|\varepsilon_{K}^{'}$, we have
$
    \|\mathbf{B}(\bar{K}_{n|p_{1}}-\bar{K}_{n|p_{2}})\| \leq C_{K}\gamma^{p-n+W}+\varepsilon_{K}.
$
Thus, $\|\theta_{\tau+1|p_{1},p_{2}}\|
    \leq C_{fb}^{2}[\frac{C_{K}\gamma^{p+W}\eta^{\tau}}{1-\frac{1}{\gamma}}(1-(\frac{1}{\gamma})^{\tau+1})+ \varepsilon_{K}\tau\eta^{\tau}].$
Choosing $\tau = t,p_{1} = t$ and $p_{2} = T$, this results in $\|\theta_{t|t,T}\| \leq C^{2}_{fb}\|\bar{x}_{1}\|[\frac{C_{K}\gamma^{1+W}\eta^{t}}{\gamma-1}(\gamma^{t}-1) + \varepsilon_{K}t\eta^{t}].$

Moreover, $\|\theta_{i|i-1,i}\| \leq C_{fb}^{2}\|\bar{x}_{1}\|[\frac{C_{K}\eta^{i-1}\gamma^{W}}{\gamma-1}(\gamma^{i}-1) + \varepsilon_{K}(i-1)\eta^{i-1}].$
Similar to the argument following \cite[Lemma 10, (25)]{chen_regret_2023}, we can find $q,C_{q}\text{ }(C_{q}>0,0<q<1)$ such that $\|(A+\mathbf{B}\bar{K})^{t-i}\| \leq C_{q}q^{t-i}$. Concluding the above, we have
\begin{align*}
    \|x_{t}-x_{t|t}\| &\leq \sum_{\tau=1}^{t}\|(A+\mathbf{B}\bar{K})^{t-i}\theta_{i|i-1,i}\|\\
    &\leq C_{fb}^{2}\|\bar{x}_{1}\|q^{t}[\frac{C_{K}\gamma^{W}}{\gamma-1}\bigg(\frac{1-(\frac{\eta\gamma}{q})^{t}}{1-\frac{\eta\gamma}{q}} - \frac{1-(\frac{\eta}{q})^{t}}{1-\frac{\eta}{q}} \bigg)\\
    &\qquad+\varepsilon_{K}\bigg(\frac{(t-1)(\frac{\eta}{q})^{t+1}-t(\frac{\eta}{q})^{t}+\frac{\eta}{q}}{(1-\frac{\eta}{q})^{2}}\bigg)].
\end{align*}
\end{proof}

Before presenting the Cost Difference Lemma that is essential for the proof of Theorem~\ref{thm:main}, we introduce the following notation.

Consider any policies $(\pi_{t}^{i})_{i=1,t=1}^{2,T-1}$ and $(\tilde{\pi}_{t}^{i})_{i=1,t=1}^{2,T-1}$ where $\pi_{t}^{i},\tilde{\pi}_{t}^{i} \in \Lambda$. We state the convention $\bar{\Pi}_{t}^{t_{0}} := (\Pi_{\tau})_{\tau=t}^{t_{0}}$ for the indices $t$ where $t\in \mathbb{N}_{T-1}$.
We again define 

\begin{align*}
&g_{t}^{i}(x_{t}, u_{t}^{1}, u_{t}^{2}) := g_{t}^{i}(x_{t}, \mathbf{u}_{t}) = x_{t}^{\transpose}Q_{t}x_{t} + \mathbf{u}_{t}^{\transpose}R_{t}^{i},\\
\begin{split}
    &V_{i,T,t}^{\bar{\Pi}_{t}^{T-1}}(x_{t}) := \\
    &\begin{cases}
        \sum_{l=0}^{T-1-t} g_{t+l}^{i}(x_{t+1+l}^{\bar{\Pi}_{t+l}^{t+l+1}}(x_{t+l}), \Pi_{t+l}(x_{t+l})) & \text{$1 \leq t < T-1$}\\
        $0$ & \text{$t \geq T-1$},
    \end{cases}
\end{split}
    \\
    &Q_{i,T,t}^{\bar{\Pi}_{t}^{T-1}}(x_{t},u_{t}^{1},u_{t}^{2}) := g_{t}^{i}(x_{t},u_{t}^{1},u_{t}^{2}) + V_{i,T,t+1}^{\bar{\Pi}_{t+1}^{T-1}}(x_{t+1}^{\bar{\Pi}_{t+1}^{T-1}}).
\end{align*}
for $i \in \{1,2\}$ where $u_{t}^{i} = \pi_{t}^{i}(x_{t})$ .
Now we present our Cost Difference Lemma.
\begin{lemma}[Cost Difference Lemma]\label{lemma:costDifference}
Given a positive integer $T \geq 1$, for $t\in \mathbb{N}_{T-1}$ and $i = \{1,2\}$, consider policies
$(\pi_{t}^{i})_{t=1}^{T-1}$,$(\tilde{\pi}_{t}^{i})_{t=1}^{T-1}$ such that $\pi_{t}^{i},\tilde{\pi}_{t}^{i} \in \Lambda$. Let $(\Pi_{t})_{t=1}^{T-1} = (\pi_{t}^{1},\pi_{t}^{2})_{t=1}^{T-1}$ and $(\tilde{\Pi}_{t})_{t=1}^{T-1} = (\tilde{\pi}_{t}^{1},\tilde{\pi}_{t}^{2})_{t=1}^{T-1}$
Then, we have
\begin{align}
    &J_{i,T}(\bar{x}_{1},(\Pi_{t})_{t=1}^{T-1}) - J_{i,T}(\bar{x}_{1},(\tilde{\Pi}_{t})_{t=1}^{T-1})\notag \\
    &= \sum_{t=1}^{T-1} Q_{i,T,t}^{\bar{\Pi}_{t}^{T-1}}(x_{t},u_{t}^{1},u_{t}^{2}) -  V_{i,T,t}^{\bar{\Pi}_{t}^{T-1}}(x_{t}) \label{eq:cost_diff_lemma},
\end{align}
where $\tilde{x}_{t+1} = \tilde{x}_{t+1}^{\bar{\Pi}_{t}^{t+1}}(\tilde{x}_{t})$, $u_{t}^{i} = \pi_{t}^{i}(x_{t})$ and $x_{t},u_{t}^{1},u_{t}^{2}$ satisfy \eqref{eq:linsys}.
\end{lemma}
\begin{proof} Starting from the RHS of \eqref{eq:cost_diff_lemma} we have:
\begin{align*}
    &\sum_{t=1}^{T-1} Q_{i,T,t}^{\bar{\Pi}_{t}^{T-1}}(x_{t},u_{t}^{1},u_{t}^{2}) -  V_{i,T,t}^{\bar{\Pi}_{t}^{T-1}}(x_{t}) \\
    &= \sum_{t=1}^{T-1} g_{t}(x_{t},u_{t}^{1},u_{t}^{2})+ V_{i,T,t+1}^{\bar{\tilde{\Pi}}_{t+1}^{T-1}}(x_{t+1})-V_{i,T,t}^{\bar{\tilde{\Pi}}_{t}^{T-1}}(x_{t})\\
    &= \underbrace{\sum_{t=1}^{T-1} g_{t}(x_{t},u_{t}^{1},u_{t}^{2})}_{J_{i,T}(\bar{x}_{1},(\Pi_{t})_{t=1}^{T-1})} - V_{i,T,1}^{\bar{\tilde{\Pi}}_{1}^{T-1}}(x_{1})\\
    &= J_{i,T}(\bar{x}_{1},(\Pi_{t})_{t=1}^{T-1}) - J_{i,T}(\bar{x}_{1},(\tilde{\Pi}_{t})_{t=1}^{T-1}).
\end{align*}
\end{proof}

\begin{proposition}\label{corollary:Delta}
    For $t\in \mathbb{N}_{T-1}$, and $i \in \{1,2\}$, consider $R_{t}^{i}$ from \eqref{eq:rblock}, $A$ and $\mathbf{B}$ from \eqref{eq:linsys}, $P_{t+1}^{i}$ from \eqref{eq:pIteration} and $\bar{K}_{t}$ from Lemma \ref{lemma:boundedP}. There exist positive scalars $\Delta_{1}$ and $\Delta_{2}$ that are invariant to time and player indices, satisfy
    \begin{align}
    \label{eq:delta1}
        &\|R_{t}^{i}+\mathbf{B}^{\transpose}P_{t+1}^{i}\mathbf{B}\| \leq \Delta_{1},\\
        \label{eq:delta2}
        &\|(R_{t}^{i}+\mathbf{B}^{\transpose}P_{t+1}^{i}\mathbf{B})\bar{K}_{t}+\mathbf{B}^{\transpose}P_{t+1}^{i}A\| \leq \Delta_{2}.
    \end{align}
\end{proposition}
 \begin{proof}
     If costs $\{R_{t}^{i}\}_{i=1,t=1}^{2,T-1}$ and $\{Q_{t}\}_{t=1}^{T}$ from LQ-DFG is an LQ-DFPG. By Lemma \ref{lemma:gamePGrelation}, there exists a scalar $\Delta^{'}$ that is independent of $t$ and $i$ that$
         \|\contB^{\transpose}P_{t+1}^{i}\| = \|\contB^{\transpose}\bar{P}_{t+1}\| \leq \Delta^{'}$, where $\bar{P}_{t+1}$ is defined in Lemma \ref{lemma:gamePGrelation}. 
     Due to Assumption \ref{assumption:bounds}, matrix $\|R_{t}^{i}\|$ is upper bounded uniformly w.r.t. $t$ and $i$. Thus, we can find such $\Delta_{1}$ and $\Delta_{2}$ to satisfy  \eqref{eq:delta1} and \eqref{eq:delta2}.
 \end{proof}
 We have established all essential auxiliary lemmas and we can now use them to prove Theorem~\ref{thm:main}.
\section{Proof of Theorem~\ref{thm:main}}
We can now prove Theorem~\ref{thm:main} by using the above auxiliary results.
Let $(\Pi_{t})_{t=1}^{T-1}$ be the control policies given in \eqref{eq:policy}, and $(\Pi_{t}^{i*})_{t=1}^{T-1}$ be the control policies that generate feedback Nash equilibrium.
By applying the Cost Difference Lemma (Lemma \ref{lemma:costDifference}), we have
\begin{align*}
    &\text{PoU}_{T}((\Pi_{t})_{t=1}^{T-1})\\
    &= \frac{1}{2}\sum_{i=1}^{2}\sum_{t=1}^{T-1} Q_{i,T,t}^{\bar{\Pi}_{t}^{*T-1}}(x_{t},u_{t}^{1},u_{t}^{2}) -  V_{i,T,t}^{\bar{\Pi}_{t}^{*T-1}}(x_{t})\\
    &= \frac{1}{2}\sum_{i=1}^{2}\sum_{t=1}^{T-1} (\mathbf{u}_{t}-\contTilde{u}_{t})^{\transpose}(R_{t}^{i}+\mathbf{B}^{\transpose}P_{t+1}^{i}\mathbf{B})(\mathbf{u}_{t}-\contTilde{u}_{t})\\
    &\qquad + 2(\mathbf{u}_{t}-\contTilde{u}_{t})^{\transpose}\bigg((R_{t}^{i}+\mathbf{B}^{\transpose}P_{t+1}^{i}\mathbf{B})\bar{K}_{t}+\mathbf{B}^{\transpose}P_{t+1}^{i}A\bigg)x_{t}.
\end{align*}

We start with finding an upper bound of the PoU. Since $\|\mathbf{u}_{t} - \contTilde{u}_{t}\| \leq \|\bar{K}_{t}-\bar{K}\|\|x_{t\mid t}-x_{t}\|  + \|\bar{K}_{t|t}-\bar{K}_{t}\| \|x_{t|t}\|$, and the elementary inequality $(a_{1}+a_{2})^2 \leq 2(a_{1}^{2}+a_{2}^{2})$ for any $a_{1}, a_{2} \in \mathbb{R}$, we have
\begin{align*}
    \|\mathbf{u}_{t} - \contTilde{u}_{t}\|^{2}
    &= \|\bar{K}x_{t} + (\bar{K}_{t\mid t}-\bar{K})x_{t\mid t} - \bar{K}_{t}x_{t}\|^{2}\\
    &\leq 2(\|\bar{K}_{t}-\bar{K}\|^{2}\|x_{t\mid t}-x_{t}\|^{2}  + \|\bar{K}_{t|t}-\bar{K}_{t}\|^{2} \|x_{t|t}\|^{2}),
\end{align*}

By Proposition~\ref{corollary:Delta}, there exist $\Delta_{1}>0,\Delta_{2}>0$, such that
\begin{align}
    &\text{PoU}((\Pi_{t})_{t=1}^{T-1})\notag\\ 
    &\leq 2\Delta_{1} \sum_{t=1}^{T-1} \bigg(\|\bar{K}_{t}-\bar{K}\|^{2}\|x_{t\mid t}-x_{t}\|^{2}  + \|\bar{K}_{t|t}-\bar{K}_{t}\|^{2} \|x_{t|t}\|^{2} \bigg) \notag\\
    &\qquad + \Delta_{2}\sum_{t=1}^{T-1}\bigg( \|\bar{K}_{t}-\bar{K}\|\|x_{t\mid t}-x_{t}\|  + \|\bar{K}_{t|t}-\bar{K}_{t}\| \|x_{t|t}\|\bigg)\notag \\
    &\qquad\qquad\bigg(\|x_{t}-x_{t|t}\| + \|x_{t|t}\|\bigg), \label{eq:RegredUnsimplified}
\end{align}
where $\Delta_{1} = \max_{t} \frac{1}{2}\|\sum_{i=1}^{2}(R_{t}^{i}+\mathbf{B}^{\transpose}P_{t+1}^{i}\mathbf{B})\|$,
    $\Delta_{2} = \max_{t} \|\sum_{i=1}^{2}\bigg( (R_{t}^{i}+\mathbf{B}^{\transpose}P_{t+1}^{i}\mathbf{B})\bar{K}_{t}+\mathbf{B}^{\transpose}P_{t+1}^{i}A\bigg)\|$.
    Applying Lemma \ref{lemma:distanceXt} and using the same $q$ as in this lemma, we have
\begin{align*}
    \|x_{t}-x_{t|t}\|&\leq C_{fb}^{2}\|\bar{x}_{1}\|q^{t}[\frac{C_{K}\gamma^{W}}{\gamma-1}\bigg(\frac{1-(\frac{\eta\gamma}{q})^{t}}{1-\frac{\eta\gamma}{q}} - \frac{1-(\frac{\eta}{q})^{t}}{1-\frac{\eta}{q}} \bigg)\\
    &\qquad+\varepsilon_{K}\bigg(\frac{(t-1)(\frac{\eta}{q})^{t+1}-t(\frac{\eta}{q})^{t}+\frac{\eta}{q}}{(1-\frac{\eta}{q})^{2}}\bigg)]\\
    &\leq C_{x}\|\bar{x}_{1}\|q^{t}.
\end{align*}
To simplify the presentation of the main result, define
\begin{align*}
    &C_{*} := \max_{t} \|\bar{K}_{t}-\bar{K}\|,C_{x} := C_{fb}^{2}D_{K}(\gamma^{W},\varepsilon_{K}),\\
    &D_{K}(\gamma^{W},\varepsilon_{K}) := \frac{C_{K}\gamma^{W}q\eta(\gamma-1)}{(q-\eta\gamma)(q-\eta)} + \frac{\varepsilon_{K}q\eta}{(q-\eta)^{2}},\\
    &\Delta_{a}(\gamma^{W},\varepsilon_{K}) := 2(\Delta_{1}+\Delta_{2})C_{*}^{2}D_{K}(\gamma^{W},\varepsilon_{K})C_{x}^{2},\\
    &\Delta_{b}(\gamma^{W},\varepsilon_{K}) := 4\Delta_{1}C_{fb}^{2}(C_{K}^{'2}\gamma^{2W}+\frac{\varepsilon_{K}^{2}}{\|\contB\|^{2}} )+2\Delta_{2}(C_{K}^{'}\gamma^{W}+\frac{\varepsilon_{K}}{\|\contB\|}),\\
    &\Delta_{c}(\gamma^{W},\varepsilon_{K}) := 2C_{x}C_{fb}\Delta_{2}(C_{K}^{'}\gamma^{W}+\frac{\varepsilon_{K}}{\|\contB\|}+C_{*}D_{K}(\gamma^{W},\varepsilon_{K})),
\end{align*}
where constants $C^{'}_{K}, \varepsilon_{K}$, and $C_{fb}$ are from Lemmas \ref{lemma:boundedPK} to \ref{lemma:distanceXt}.

To simplify \eqref{eq:RegredUnsimplified}, let
\begin{align*}
    \Gamma_1(t) &:= C_{*}^{2}D_{K}(\gamma^{W},\varepsilon_{K})C_{x}^{2}q^{2t} + 2(C_{K}^{'2}\gamma^{2W}+\varepsilon_{K}^{'2})C_{fb}^{2}\eta^{2t} \\
    \Gamma_2(t) &:= C_{*}D_{K}(\gamma^{W},\varepsilon_{K})C_{x}q^{t}+(C_{K}^{'}\gamma^{W}+\varepsilon_{K}^{'})C_{fb}\eta^{t},\\
    \Gamma_3(t) &:= C_{x}q^{t}+C_{fb}\eta^{t}
\end{align*}

From Lemmas \ref{lemma:boundedPK} and \ref{lemma:multGain}, we have 
\begin{align*}
     \|\bar{K}_{t}-\bar{K}\|^{2}\|x_{t\mid t}-x_{t}\|^{2}  + \|\bar{K}_{t|t}-\bar{K}_{t}\|^{2} \|x_{t|t}\|^{2} & \leq \|\bar{x}_1\|^2\Gamma_1(t) \\
     \|\bar{K}_{t}-\bar{K}\|\|x_{t\mid t}-x_{t}\|  + \|\bar{K}_{t|t}-\bar{K}_{t}\| \|x_{t|t}\| & \leq \|\bar{x}_{1}\| \Gamma_2(t)\\
     \|x_{t}-x_{t|t}\| + \|x_{t|t}\| & \leq \|\bar{x}_{1}\|\Gamma_3(t).
\end{align*}
where $\eta$ and $\gamma$ are defined in aforementioned lemmas.
Moreover, 
\begin{align*}
   \sum_{t=1}^{T-1} \Gamma_1(t) \leq \overline{\Gamma}_1,
   \sum_{t=1}^{T-1} \Gamma_2(t)\Gamma_3(t) \leq \overline{\Gamma}_2,
 \end{align*}
where 
\begin{align*}
    \overline{\Gamma}_1 &= C_{*}^{2}D_{K}(\gamma^{W},\varepsilon_{K})C_{x}^{2}\frac{1-q^{2T}}{1-q^{2}} + 2C_{fb}^{2}(C_{K}^{'2}\gamma^{2W}+\varepsilon_{K}^{'2})\frac{1-\eta^{2T}}{1-\eta^{2}}\\
    \overline{\Gamma}_2 &=C_{*}D_{K}(\gamma^{W},\varepsilon_{K})C_{x}^{2}\frac{1-q^{2T}}{1-q^{2}}+ C_{fb}^{2}(C_{K}^{'}\gamma^{W}+\varepsilon_{K}^{'})\frac{1-\eta^{2T}}{1-\eta^{2}}\\
    & \qquad +[D_{K}(\gamma^{W},\varepsilon_{K})C_{x}C_{fb}+C_{x}C_{fb}(C_{K}^{'}\gamma^{W}+\varepsilon_{K}^{'})]\frac{1-(q\eta)^{T}}{1-q\eta}
\end{align*}
Note that the right hand side of \eqref{eq:RegredUnsimplified} then can be bounded by $\overline{\Gamma}$ where
\begin{align*}
    \overline{\Gamma} &= \|\bar{x}_1\|^2 \left (2\Delta_1 \overline{\Gamma}_1 + \Delta_2 \overline{\Gamma}_2 \right )\\
    &= \|\bar{x}_{1}\|^{2}\left [\Delta_{a}(\gamma^{W},\varepsilon_{K}^{'})\frac{1-q^{2T}}{1-q^{2}} +  \Delta_{b}(\gamma^{W},\varepsilon_{K}^{'})\frac{1-\eta^{2T}}{1-\eta^{2}} \right .\\ 
    & \left . \qquad + \Delta_{c}(\gamma^{W},\varepsilon_{K}^{'})\frac{1-(q\eta)^{T}}{1-q\eta}\right ]
\end{align*}
By inspection, the PoU upper bound, $\overline{\Gamma}$, can be expressed as
$C_{1}\gamma^{2W}+C_{2}\gamma^{W}+C_{3}\varepsilon_{K}$, where $C_{1},C_{2}$ and $C_{3}$ are
monotonically increasing with respect to $\|\bar{x}_1\|$, $\lambda_{max}(R_{max}^{'})$
and $\lambda_{max}(Q_{max})$, and the inverse of $\rho(A+\contB\bar{K})$,
$\lambda_{min}(R_{min}^{'})$ and $\lambda_{min}(Q_{min})$.

We next find the lower bound of PoU.
By Assumption \ref{assumption:bounds}, $R_{t}^{i} \in \mathbb{S}_{++}^{m}$. From \eqref{eq:pIteration} and Assumption \ref{assumption:bounds}, $Q_{t} \in \mathbb{S}_{++}^{n}$ and $R_{t}^{i} \in \mathbb{S}_{++}^{m}$, we can deduce that $P_{t}^{i} \in \mathbb{S}_{++}^{n}$ by induction. Therefore, we have
$R_{t}^{i}+\mathbf{B}^{\transpose}P_{t+1}^{i}\mathbf{B} \in \mathbb{S}_{++}^{n}$. Equivalently, $(\mathbf{u}_{t}-\contTilde{u}_{t})^{\transpose}(R_{t}^{i}+\mathbf{B}^{\transpose}P_{t+1}^{i}\mathbf{B})(\mathbf{u}_{t}-\contTilde{u}_{t}) \geq 0$. Let $F_{t} := 2(\mathbf{u}_{t}-\contTilde{u}_{t})^{\transpose}\bigg((R_{t}^{i}+\mathbf{B}^{\transpose}P_{t+1}^{i}\mathbf{B})\bar{K}_{t}+\mathbf{B}^{\transpose}P_{t+1}^{i}A\bigg)x_{t}$. Note that
\begin{align*}
    \text{PoU}((\Pi_{t})_{t=1}^{T-1}) &\geq \sum_{t=1}^{T-1} F_{t}\\
    &\geq \sum_{t=1}^{T-1} -2\|\mathbf{u}_{t}-\contTilde{u}_{t}\|(R_{t}^{i}+\mathbf{B}^{\transpose}P_{t+1}^{i}\mathbf{B})\bar{K}_{t}\\
    &\qquad\quad+\mathbf{B}^{\transpose}P_{t+1}^{i}A\|\|x_{t}\|\\
    &\geq -\Delta_{2}\bar{\Gamma}_{2}.
\end{align*}
By inspection, the PoU lower bound $-\Delta_{2}\bar{\Gamma}_{2}$, can be expressed as $-C_{1}^{'}\gamma^{2W} - C_{2}^{'}\gamma^{W} - C_{3}^{'}\varepsilon_{K}$, where $C_{1}^{'},C_{2}^{'}$ and $C_{3}^{'}$ are positive scalars that also monotonically increasing w.r.t. $\|\bar{x}_{1}\|$, $\lambda_{max}(R_{max}^{'})$ and $\lambda_{max}(Q_{max})$, and the inverse of $\rho(A+\contB\bar{K})$, $\lambda_{min}(R_{min}^{'})$ and $\lambda_{min}(Q_{min})$.

The proposition below proves that the sufficient condition provided in Section \ref{sec:assumptions} is valid for Assumptions \ref{assumption:parameters}, \ref{assumption:controllable}, \ref{assumption:positiveRp}, and \ref{assumption:potentialRepeat}.
\begin{proposition}\label{prop:valid}
    Suppose $A$ is a full-rank square matrix. Consider $R_{t}^{i},Q_{t}$ are defined in \eqref{eq:costSelection}.
    For integers $T \geq 1$, $0 \leq t \leq T-1$ and $i,j\in \{1,2\}$, if $\frac{b_{i}^{2}}{r_{i,t}} = \frac{b_{j}^{2}}{r_{j,t}}$, then parameters $P_{t},\Theta_{t}$ defined in \eqref{eq:pIteration} and \eqref{eq:Theta}, respectively, satisfy conditions \eqref{eq:costFPDG4}-\eqref{eq:costFPDG3}, and Assumption \ref{assumption:potentialRepeat}.
\end{proposition}
\begin{proof}
    For $t = T$, since $P_{T}^{1} = P_{T}^{2} = Q_{T}$, and
    \begin{align*}
        \Theta_{T-1} = 
        \begin{bmatrix}
            r_{1,T-1} & 0\\
            0 & r_{2,T-1}
        \end{bmatrix}
        + \contB^{\transpose}Q_{T}\contB,
    \end{align*}
    we have 
    $\Theta_{T-1} \succ 0 $.
    Then, for $t = T-1$, 
    \begin{align*}
        P_{T-1}^{1}-P_{T-1}^{2} = K_{T-1}^{\transpose}(R_{T-1}^{1}-R_{T-1}^{2})K_{T-1}.
    \end{align*}
    We next claim that $P_{T-1}^{1}=P_{T-1}^{2}$.
    Since
    \begin{align*}
        &\det(\Theta_{T-1})^{2}\contB\Theta_{T-1}^{-1}(R_{T-1}^{1}-R_{T-1}^{2})\Theta_{T-1}^{-1}\contB^{\transpose} = \\
        &
        \begin{bmatrix}
            -b_{1} & -b_{2}\\
            0 & 0
        \end{bmatrix}
        \begin{bmatrix}
            r_{2,T-1}+B^{2\transpose}P_{T-1}^{2}B^{2\transpose} & B^{2\transpose}P_{T-1}^{2}B^{1\transpose}\\
            B^{1\transpose}P_{T-1}^{1}B^{2\transpose} & r_{1,T-1}+B^{1\transpose}P_{T-1}^{1}B^{1\transpose}
        \end{bmatrix}\\
        &
        \begin{bmatrix}
            r_{1,T-1} & 0\\
            0 & -r_{2,T-1}
        \end{bmatrix}
        \begin{bmatrix}
            r_{2,T-1}+B^{2\transpose}P_{T-1}^{2}B^{2\transpose} & B^{2\transpose}P_{T-1}^{2}B^{1\transpose}\\
            B^{1\transpose}P_{T-1}^{1}B^{2\transpose} & r_{1,T-1}+B^{1\transpose}P_{T-1}^{1}B^{1\transpose}
        \end{bmatrix}\\
        &\begin{bmatrix}
            -b_{1} & 0\\
            -b_{2} & 0
        \end{bmatrix}
    \end{align*}
    Due to
    \begin{align*}
        \begin{bmatrix}
            \alpha & \beta\\
            \beta & \gamma
        \end{bmatrix}
        \begin{bmatrix}
            h_{1} & 0\\
            0 & h_{2}
        \end{bmatrix}
        \begin{bmatrix}
            \alpha & \beta\\
            \beta & \gamma
        \end{bmatrix} = 
        \begin{bmatrix}
            h_{1}\alpha^{2}+h_{2}\beta^{2} & h_{1}\alpha\beta + h_{2}\beta\gamma\\
            h_{1}\alpha\beta + h_{2}\beta\gamma & h_{1}\beta^{2} + h_{2}\gamma
        \end{bmatrix},
    \end{align*}
    and
    \begin{align*}
        B^{2\transpose}P_{T-1}B^{2} = b_{2}^{2}[Q_{T}]_{11},\\
        B^{1\transpose}P_{T-1}B^{1} = b_{1}^{2}[Q_{T}]_{11},
    \end{align*}
    substitute $h_{1} = r_{1,T-1}, h_{2}=-r_{2,T-1}$, we have
    \begin{align*}
        \frac{h_{1}\alpha}{-h_{2}\gamma} = \frac{r_{1,T-1}}{r_{2,T-1}}\frac{r_{2,T-1}+b_{2}^{2}[Q_{T}]_{11}}{r_{1,T-1}+b_{1}^{2}[Q_{T}]_{11}} = \frac{r_{1,T-1}}{r_{2,T-1}}\frac{r_{2,T-1}}{r_{1,T-1}} = 1,
    \end{align*}
    therefore, $h_{1}\alpha\beta + h_{2}\beta\gamma=0$.
    Moreover,
    \begin{align*}
        &
        \begin{bmatrix}
            -b_{1} & -b_{2}\\
            0 & 0
        \end{bmatrix}
        \begin{bmatrix}
            c_{1} & 0\\
            0 & c_{2}
        \end{bmatrix}
        \begin{bmatrix}
            -b_{1} & 0\\
            -b_{2} & 0
        \end{bmatrix}
        =
        \begin{bmatrix}
            c_{1}b_{1}^{2}+c_{2}b_{2}^{2} & 0\\
            0 & 0
        \end{bmatrix}.
    \end{align*}
    Here, $c_{1} = r_{1,T-1}(r_{2,T-1}+b_{2}^{2}[Q_{T}]_{11}), c_{2} = -r_{2,T-1}(r_{1,T-1}+b_{1}^{2}[Q_{T}]_{11})$. Apply $\frac{b_{1}^{2}}{b_{2}^{2}} = \frac{r_{1,T-1}}{r_{2,T-1}}$ again, we have $c_{1}b_{1}^{2}+c_{2}b_{2}^{2}=0$. Therefore, $K_{T-1}^{\transpose}(R_{T-1}^{1}-R_{T-1}^{2})K_{T-1} = 0$. 
    
    Let's assume that $P_{t+1}^{1}-P_{t+1}^{2}=0$. Due to 
    \begin{align*}
        \Theta_{t} = 
        \begin{bmatrix}
            r_{1,t} & 0\\
            0 & r_{2,t}
        \end{bmatrix}
        + \contB^{\transpose}P^{1}_{t+1}\contB,
    \end{align*}
    we have $\Theta_{t} \succ 0$, therefore $\det(\Theta_{t})>0$.
    Consider 
    \begin{align*}
        K_{t}^{\transpose}(R_{t}^{1}-R_{t}^{2})K_{t}=
        \contB\Theta_{t}^{-1}(R_{t}^{1}-R_{t}^{2})\Theta_{t}^{-1}\contB^{\transpose}.
    \end{align*}
    By repeating the steps above and using the relationship of $B^{1\transpose}P_{t+1}^{1}B^{1}=b_{1}^{2}[P_{t+1}^{1}]_{11}$, we have
    \begin{align*}
        \frac{h_{1}\alpha}{-h_{2}\gamma} = \frac{r_{1,t}}{r_{2,t}}\frac{r_{2,t}+b_{2}^{2}[P_{t+1}^{2}]_{11}}{r_{1,t}+b_{1}^{2}[P_{t+1}^{1}]_{11}} = 1,
    \end{align*}
    and $\frac{b_{1}^{2}}{b_{2}^{2}} = \frac{r_{1,t}}{r_{2,t}}$.
    Therefore, by induction, we can conclude that for $0 \leq t \leq T-1$, parameters $P_{t},\Theta_{t}$ defined in \eqref{eq:pIteration} and \eqref{eq:Theta}, respectively, satisfy conditions \eqref{eq:costFPDG4}-\eqref{eq:costFPDG3}, and Assumption \ref{assumption:potentialRepeat}.
\end{proof}

\end{document}